\documentclass{article}
\usepackage{amsmath}
\usepackage{amssymb}
\usepackage{amsfonts}
\usepackage{amscd}
\usepackage{amsthm}
\usepackage{mathrsfs}
\usepackage[numbers]{natbib}

\newcommand{\CC}{{\rm\bf C}}
\newcommand{\RR}{{\rm\bf R}}
\newcommand{\QQ}{{\rm\bf Q}}
\newcommand{\ZZ}{{\rm\bf Z}}

\newcommand{\FF}{{\rm\bf F}}

\newcommand{\Adeles}{{\rm\bf A}}

\newcommand{\OO}{\mathcal {O}}

\DeclareMathOperator{\absNorm}{\mathfrak{N}}

\DeclareMathOperator{\diag}{\mathrm{diag}}
\DeclareMathOperator{\Gm}{\mathrm {{\bf G}_m}}
\DeclareMathOperator{\GL}{\mathrm{GL}}

\DeclareMathOperator{\End}{\mathrm {End}}
\DeclareMathOperator{\Hom}{\mathrm {Hom}}
\DeclareMathOperator{\kernel}{\mathrm {ker}}

\DeclareMathOperator{\rang}{\mathrm {rank}}

\DeclareMathOperator{\res}{\mathrm{Res}}
\DeclareMathOperator{\Ad}{\mathrm{Ad}}
\DeclareMathOperator{\ad}{\mathrm{ad}}
\DeclareMathOperator{\der}{\mathrm{der}}

\DeclareMathOperator{\Lie}{\mathrm{Lie}}
\newcommand{\zentrum}{\mathscr{C}}

\newcommand{\lieg}{{\mathfrak {g}}}

\newcommand{\liek}{{\mathfrak {k}}}

\newcommand{\lies}{{\mathfrak {s}}}

\newcommand{\liep}{{\mathfrak {p}}}

\newcommand{\liegl}{{\mathfrak {gl}}}

\DeclareMathOperator*{\Otimes}{\ensuremath{\otimes}}

\newcommand{\absnorm}[1]{\ensuremath{\left|\!\left|{#1}\right|\!\right|}}

\newcommand{\adots}{\ensuremath{
\put(-3,-6){\normalsize$\,$}
\put(-2.2,-3.5){\normalsize$\cdot$}
\cdot
\put(-1.0,3.5){\normalsize$\cdot$}
\put(0.0,6){\normalsize$\,$}
}}

\theoremstyle{plain}
\newtheorem{theorem}{Theorem}[section]
\newtheorem{lemma}[theorem]{Lemma}
\newtheorem{corollary}[theorem]{Corollary}
\newtheorem{proposition}[theorem]{Proposition}

\bibpunct{[}{]}{,}{n}{}{,}

\begin{document}

\title{On $p$-adic $L$-functions for $\GL(n)\times\GL(n-1)$ over totally real fields}
\author{Fabian Januszewski\footnote{Part of this research was conducted while the author was a guest at the Institut Poincar\'e Centre Emile Borel in Paris and the author also acknowledges financial support from the German Academy of Sciences Leopoldina grant no. LDPS 2009-23 during his stay at the University of California at Los Angeles.}}

\maketitle

\begin{abstract}
We refine and extend previous constructions of $p$-adic $L$-functions for Rankin-Selberg convolutions on $\GL(n)\times\GL(n-1)$ to the case of cuspidal regular algebraic representations of finite slope over totally real fields without any restrictions on class numbers or the residual characteristic. We also prove an intrinsic functional equation for these $p$-adic $L$-functions, which will be of interest in further study of their arithmetic properties.
\end{abstract}

{\small\tableofcontents}

\markboth{Fabian Januszewski}{Modular symbols for reductive groups and $p$-adic Rankin-Selberg convolutions}

\section*{Introduction}\label{sec:introduction}

In this paper we study the problem of $\mathfrak{p}$-adic interpolation of the special values of twisted Rankin-Selberg $L$-functions of irreducible cuspidal automorphic representations $\pi$ and $\sigma$ on $\GL_n$ and $\GL_{n-1}$ for $n\geq 2$ in the sense of \citep{jpss1983,cogdellpiatetskishapiro2004}. This is in some sense a continuation of the previous works \citep{birch1971,manin1972,mazur1972,mazurswinnertondyer1974,schmidt1993,kazhdanmazurschmidt2000,schmidt2001,kastenschmidt2008,januszewski2009}.

We review the general problem of $p$-adic interpolation and extend the results of \citep{januszewski2009}. To be more precise, let $\pi$ and $\sigma$ be irreducible cuspidal automorphic representations of $\GL_n(\Adeles_k)$ and $\GL_{n-1}(\Adeles_k)$, where $\Adeles_k$ denotes the ad\`ele ring of a totally real number field $k$. We fix a finite place $\mathfrak{p}$ of $k$.

Assuming that $\pi$ and $\sigma$ are cuspidal irreducible regular algebraic, i.e.\ occur in the cohomology of arithmetic groups, and furthermore that the pair of representations $(\pi,\sigma)$ is of finite slope at $\mathfrak{p}$, we show the existence of a $\mathfrak{p}$-adic vector-valued distribution $\mu=(\mu^\nu)_\nu$ on the ray class group $\mathcal{C}l(\mathfrak{p}^{\infty})$, which is tempered and whose order is directly related to the slope of the pair $(\pi,\sigma)$. In the slope $0$ case, i.e.\ when the pair $(\pi,\sigma)$ is ordinary at $\mathfrak{p}$, then each $\mu^\nu$ is indeed a $\mathfrak{p}$-adic measure (cf.\ Theorems \ref{thm:distribution} and \ref{thm:boundedness}). In any case $\mu$ has the property that for certain \lq\lq{}periods\rq\rq{} $\Omega^{{\rm sign}(\chi)(-1)^{\nu}}(\frac{1}{2}+\nu)\in\CC$ we have
$$
\int\limits_{\mathcal{C}l(\mathfrak{p}^{\infty})}\chi d\mu^\nu=\Omega^{{\rm sign}(\chi)(-1)^{\nu}}(\frac{1}{2}+\nu)\cdot\hat{\kappa}(\mathfrak{f})\cdot G(\chi)^{\frac{n(n-1)}{2}}\cdot L^{(\mathfrak{p})}(\frac{1}{2}+\nu,(\pi\times\sigma)\otimes\chi)
$$
for any Hecke character $\chi$ of finite order with non-trivial conductor $\mathfrak{f}\mid \mathfrak{p}^{\infty}$, and any integer $\nu$ such that $\frac{1}{2}+\nu$ is critical for $L(s,\pi\times\sigma)$ in the sense of \citep{deligne1979}. Strictly speaking $\mu$ and the above interpolation formula depend on an (ordered) choice of Hecke roots for $\pi$ and $\sigma$ at $\mathfrak{p}$. See Theorem \ref{thm:interpolation} below for the precise statement and the definition of the quantities involved.

The periods $\Omega^{{\rm sign}(\chi)}(\frac{1}{2}+\nu)$ were conjectured to be non-zero for a long time, and recently Sun gave a proof of the non-vanishing for general $n\geq 2$ (cf.\ \cite{januszewskisunpre}). Previous results were \cite{kastenschmidt2008} for $n=3$ and Mazur \cite{schmidt1993} for $n=2$. Thanks to Sun's breakthrough our results are unconditional.

Even over $\QQ$ and $n=3$ our result is new, as only the construction of $p$-adic $L$-functions for representations supporting cohomology with trivial coefficients had been carried out so far, cf.\ \cite{schmidt1993,januszewski2009}.

In the general case our construction overcomes the restrictions on class numbers faced in \citep{januszewski2009}. We point out that we have no restriction on the residue characteristic of $\mathfrak{p}$. In particular our results cover characteristic $2$ as well.

We give a purely cohomological construction of the distribution and proof a functional equation for the resulting multi-valued $\mathfrak{p}$-adic $L$-function (cf.\ Theorem \ref{thm:functionaleq} below). The proof is however more involved than in the classical case $n=2$, which was the only previously known case in our setting.

The different components of this $\mathfrak{p}$-adic $L$-function should be mutually related by a Manin's trick type argument. In our case however the involved finite-dimensional representations (i.e.\ the coefficients of cohomology) are much more involved, and up to present we were unable to make Manin's argument work in this setting.

Our method readily generalizes to $p$-adic instead of merely $\mathfrak{p}$-adic interpolation, once the evaluation of the modified local zeta-integral at $\mathfrak{p}$ at the trivial character is known. This is the case for $n=2$ classically and for $n=3$ due to yet unpublished work by Denis Ungemach. However we prefer to fix a single place, as this avoids leaving out infinitely many interpolation values where the precise interpolation formula is still unknown, and also simplifies notation.

Although we don't use automorphic symbols here explicitly, this work was influenced by \citep{dimitrov2011}, where the case $n=2$ is treated with respect to additional non-abelian variables and thereby and a strong connection between $\mathfrak{p}$-adic $L$-functions and the corresponding Galois deformations is established.

In a future paper we will apply the results obtained here to the construction of $p$-adic $L$-functions for families of $p$-ordinary automorphic representations on $\GL(n)\times \GL(n-1)$ in the sense of Hida (cf.\ \cite{hida1995}).

\paragraph{{\bf Acknowledgements.}}
The author thanks Binyong Sun for helpful remarks, and Mladen Dimitrov for having provided him with a preliminary version of his article \citep{dimitrov2011}. The author is also grateful to the Institut Poincar\'e in Paris and the number theory group at UCLA for their hospitality and good working conditions. Finally the author thanks Haruzo Hida for his continuous interest in this work.

\section*{{\rm\em Notation}}

Let $G$ be a topological group. Then $G^0$ denotes the connected component of the unit $1\in G$, the same notation applies to algebraic groups with respect to the Zariski topology.

Denote by $G^{\der}$ the commutator group of a linear algebraic group $G$. Here and in the sequel $\Lie(G)$ is the Lie algebra of $G$. The differential of a morphism $f:G\to H$ of linear algebraic (or of Lie) groups is denoted by $L(f)$. $\mathscr R(G)$ (resp.\ $\mathscr R_{\rm u}(G)$) is the (unipotent) radical of $G^0$. If $G$ is defined over a number field $k$, we usually $\mathscr X$ denotes a (connected) symmetric space for $G(k_\RR)$, $k_\RR:=k\otimes_\QQ\RR$, and reserve supscript \lq$\ad$\rq{} to denote $\mathscr X^{\ad}:=\mathscr X/S(\RR)^0$ for the maximal $\QQ$-split torus $S$ in the radical of $G$.

For a global field $k$, we write $\OO_k$ for its ring of integers. We write $\absNorm(\mathfrak{a})=[\OO_k:\mathfrak{a}]$ for any fractional ideal $0\neq\mathfrak{a}\subseteq k$. We denote by $k_\mathfrak{p}$ the completion of $k$ at the place $\mathfrak{p}$ and by $\OO_{k,\mathfrak{p}}$ its valuation ring. Usually $p$ is its residual characteristic. We write $\Adeles_k$ resp.\ $\Adeles_k^{(\infty)}$ for the ring of (finite) ad\`eles over $k$. For a place $v$ of $k$ we let $\Adeles_k^{(\infty v)}$ denote the ring of finite ad\`eles with the $\infty$- and $v$-component removed.

For a quasi-character $\chi:k_{\mathfrak{p}}^\times\to\CC^\times$ of conductor $\mathfrak{f}=\OO_{k,\mathfrak{p}}f$ and an additive unramified character $\psi:k\to\CC^\times$ we fix the $\psi$-Gau\ss{} sum as
$$
G(\chi):=\sum_{a\!\!\pmod{\mathfrak{f}}}\chi\left(\frac{a}{f}\right)\psi\left(\frac{a}{f}\right).
$$
This is independent of $f$ and differs from the notion in \citep{kazhdanmazurschmidt2000,schmidt2001,januszewski2009} by the factor $\chi(f)$. This notion of Gau\ss{} sum naturally globalizes. We define
$$
t_{(f)}:=\diag(f^{n-1},f^{n-2},\dots,1)\in\GL_n(k_\mathfrak{p}).
$$
Write $w_{n}$ for the longest element of the Weyl group in $\GL_n$ (realized as permutation matrices). Define
$$
h^{(1)}:=
\begin{pmatrix}
&&&1\\
&w_{n-1}&&\vdots\\
&&&\vdots\\
0&\hdots&0&1
\end{pmatrix}\in\GL_{n}(\ZZ),
$$
and for any $f\in k_\mathfrak{p}^\times$ set
$$
h^{(f)}:=t_{(f)}^{-1}\cdot h^{(1)}\cdot t_{(f)}\in\GL_n(k_{\mathfrak{p}}).
$$
Throughout the paper we fix the diagonal embedding $j:\GL_{n-1}\to\GL_n$ by
$$
g\mapsto\begin{pmatrix}g&0\\0&1\end{pmatrix}.
$$

\section{Hecke algebras and Hecke relations}\label{sec:hecke}

In this section we review the interrelation of the Hecke algebra of full level and the Hecke algebra of Iwahori level. We need to extend the previous study in \citep{schmidt1993,kazhdanmazurschmidt2000,schmidt2001,januszewski2009} slightly, as due to our possibly non-trivial central characters the action of the center matters.

\subsection{Hecke algebras of finite level}

For the theory of parabolic Hecke algebras we refer to \citep{gritsenko1992} and for an overview of what we use see \citep{kazhdanmazurschmidt2000,januszewski2009}.

For any Hecke pair $(R,S)$ we define the free $\ZZ$-module $\mathcal H_\ZZ(R,S)$ over the set of all double cosets $RsR$, which naturally embeds into the free $\ZZ$-module $\mathscr R_\ZZ(R,S)$ over the set of the right cosets $sR$, $s\in S$, by the coset decomposition
$$
RsR=\bigsqcup_i s_iR\mapsto\sum_i s_iR.
$$
We identify $\mathcal H_\ZZ(R,S)$ with its image under this embedding, i.e.\ $\mathcal H_\ZZ(R,S)$ is the $\ZZ$-module of $R$-invariants under the action
$$
R\times \mathscr R_\ZZ(R,S)\to \mathscr R_\ZZ(R,S),\;\;\;(r,sR)\mapsto rsR.
$$
Furthermore $\mathcal H_\ZZ(R,S)$ admits a structure of an associative $\ZZ$-algebra with the multiplication
$$
\left(\sum_i s_iR\right)\cdot\left(\sum_jt_jR\right):=\sum_{i,j}s_it_jR.
$$
This algebra is unitary if and only if $R\cap S\neq\emptyset$. For any commutative ring $A$ we set
$$
\mathcal H_A(R,S):=\mathcal H_\ZZ(R,S)\otimes_\ZZ A.
$$
$\mathcal H_A(R,S)$ is an associative algebra over $A$. We define the {\em Hecke algebra} of the pair $(R,S)$ by $\mathcal H(R,S):=\mathcal H_\CC(R,S)$.

For a locally compact topological group $G$ and an compact open subgroup $K\leq G$ the module $\mathscr R_A(K,G)$ may be interpreted as the $A$-module of locally constant right $K$-invariant mappings $f:G\to A$ with compact support and $\mathcal H_A(K,G)$ is just the submodule of left $K$-invariant mappings. In this language multiplication is given by convolution
$$
\alpha*\beta\;:\;
x\mapsto \int_G \alpha(g)\beta(xg^{-1})dg,
$$
where $dg$ is the right invariant Haar measure on $G$ which assigns measure $1$ to $K$. This integral is eventually a finite sum with integer coefficients, hence this interpretations is valid even without assuming $A\subseteq\CC$.

All Hecke algebras we consider arise in this topological context. We have the elementary 
\begin{proposition}\label{prop:heckeinjektion}
Let $G$ denote a locally compact group, $H\leq G$ a closed subgroup and let $K\leq G$ be a compact open subgroup such that $L= H\cap K$ and $HK=G$. Then the restriction
$$
\alpha\mapsto \alpha|_H
$$
defines a monomorphism $\mathcal H_A(K,G)\to\mathcal H_A(L,H)$ of $A$-algebras.
\end{proposition}

\subsection{The standard Hecke algebra}

Fix a global field $k$ and a finite place $\mathfrak{p}$ of $k$. For the standard Hecke algebra $\mathcal H_\CC(\GL_n(\OO_{k,\mathfrak{p}}),\GL_n(k_\mathfrak{p})))$ at $\mathfrak{p}$, i.e.\ $K=\GL_n(\OO_{k,\mathfrak{p}})$, and $G=\GL_n(k_{\mathfrak{p}})$ we have the Satake isomorphism
$$
\mathcal S:\mathcal H_\CC(K,G)\to \CC[X_1^{\pm1},\dots,X_n^{\pm1}]^{S_n},
$$
$$
T_\nu\mapsto \absNorm(\mathfrak{p})^{\frac{\nu(\nu+1)}{2}}\cdot\sigma_\nu(X_1,\dots,X_n),\;\;\;(0\leq\nu\leq n)
$$
where $S_n$ is the symmetric group, permutating the $X_i$, and
$$
T_\nu:=K\begin{pmatrix}\varpi\cdot{\bf1}_{n}&0\\0&{\bf1}_{n-\nu}\end{pmatrix}K
$$
is independent of the choice of a prime $\varpi$. Furthermore $\sigma_\nu$ is the elementary symmetric polynomial of degree $\nu$ in $X_1,\dots,X_n$, cf. \citep{tamagawa1963,satake1963}.

\subsection{The Hecke algebras of Iwahori level}

For an integer $r>0$ and fix the Iwahori subgroup $K_{I^{(r)}}:=I_n^{(r)}\subseteq K$ of level $\mathfrak{p}^r$ as the subgroup of matrices becoming upper triangular mod $\mathfrak{p}^r$. We let $T^+$ denote the monoid of integral diagonal matrices $\diag(a_1,\dots,a_n)$ satisfying the dominance condition
$$
|a_1|_\mathfrak{p}\leq
|a_2|_\mathfrak{p}\leq\cdots\leq
|a_n|_\mathfrak{p}.
$$
It is easy to see that $\Delta_{I^{(r)}}:=K_{I^{(r)}}T^+K_{I^{(r)}}$ is a closed submonoid of $G$.

We have $G=\Delta_{I^{(r)}}K$ due to the Iwasawa decomposition (see below), and furthermore $I^{(r)}\cap K=K_{I^{(r)}}$. Hence the above Proposition applies and shows that we have a canonical inclusion
$$
\mathcal H_G:=\mathcal H_\QQ(K,G)\to\mathcal H_\QQ(K_{I^{(r)}},\Delta_{I^{(r)}})=:\mathcal H_I
$$
of the standard Hecke algebra into the Hecke algebra of Iwahori level. We will study the latter Hecke algebra by means of the parabolic Hecke algebra.

\subsection{The parabolic Hecke algebra}

We define $K_B:=B\cap K=B_n(\OO_{k,\mathfrak{p}})$ and the {\em parabolic Hecke algebra} as $\mathcal H_B:=\mathcal H_\QQ(K_B,B)$. Then Iwasawa decomposition \citep{iwahorimatsumoto1965}, Proposition 2.33, \citep{satake1963}, section 8.2, guarantees that the hypothesis of Proposition \ref{prop:heckeinjektion} is fulfilled and we see that $\mathcal H_B$ is a ring extension of $\mathcal H_G$, with respect to the explicit embedding $\epsilon:\mathcal H_G\to\mathcal H_B$ given by
$$
\sum_i a_i\cdot g_iK\mapsto \sum_i a_i\cdot g_iK_B,
$$
where $g_i\in B$. This embedding factors over the Hecke algebra of Iwahori level.

In general $\mathcal H_B$ is a huge non-commutative algebra, containing a well behaved commutative subalgebra that was studied by Gritsenko. This will help us to understand (a corresponding commutative subalgebra of) the Hecke algebra of Iwahori level.

\subsection{Decomposition of Hecke polynomials}

We restrict our attention to the finitely generated subalgebra $\mathcal H_B^0$ of $H_B$ generated by
$$
U_i:=K_{B}
\begin{pmatrix}
{\bf1}_{i-1}&0&0\\
0&\varpi&0\\
0&0&{\bf1}_{n-i}
\end{pmatrix}
K_{B},
$$
which commute in $\mathcal H_B$ by \citep[Lemma 2]{gritsenko1992} and hence $\mathcal H_{B}^0$ is commutative. This algebra contains $\mathcal H_I$ and following \citep[Theorem 2]{gritsenko1992} we have over $\mathcal H_B^0$ a decomposition of the Hecke polynomial
$$
H_\mathfrak{p}(X):=\sum_{\nu=0}^n (-1)^\nu \absNorm(\mathfrak{p})^{\frac{(\nu-1)\nu}{2}}T_\nu X^{n-\nu}\in\mathcal H_I(X)
$$
into linear factors
\begin{equation}
H_\mathfrak{p}(X)=\prod_{i=1}^n(X-U_i).
\label{eq:gritsenko}
\end{equation}
The unit element of $\mathcal H_{B}^0$ is
$$
V_{\mathfrak{p},0}:=K_{B}{\bf1}_nK_{B}
$$
We define for $1\leq \nu\leq n$ the operators
$$
V_{\mathfrak{p},\nu}:=\absNorm(\mathfrak{p})^{-\frac{(\nu-1)\nu}{2}}\cdot U_1 U_2\cdots U_\nu\in\mathcal H_{B}
$$
and
$$
V_{\mathfrak{p}}:=\prod_{\nu=1}^{n-1}V_{\mathfrak{p},\nu}\in\mathcal H_{B}.
$$
We also introduce
$$
V_{\mathfrak{p}}':=V_{\mathfrak{p},n}\cdot V_\mathfrak{p}=\prod_{\nu=1}^{n}V_{\mathfrak{p},\nu}
\in\mathcal H_{B}
$$
\begin{lemma}\label{lem:hecke1}
For $0\leq \nu\leq n$ the operators $V_{\mathfrak{p},\nu}$ lie in $H_{I^{(r)}}$ and
$$
V_{\mathfrak{p},\nu}=
K_{I^{(r)}}
\begin{pmatrix}
\varpi\cdot {\bf1}_{\nu}&0\\
0& {\bf1}_{n-\nu}
\end{pmatrix}
K_{I^{(r)}}
=
\bigsqcup_A
\begin{pmatrix}
\varpi\cdot {\bf1}_{\nu}& A\\
0& {\bf1}_{n-\nu}
\end{pmatrix}
K_{I^{(r)}},
$$
where $A\in \OO_{k,\mathfrak{p}}^{\nu\times n-\nu}$ runs through a system of representatives modulo $\mathfrak{p}$. Furthermore the Hecke operators $V_{\mathfrak{p},\nu}$ commute for $0\leq\nu\leq n$ and
$$
V_{\mathfrak{p}}=
K_{I^{(r)}} t_{(\varpi)}K_{I^{(r)}}=
\bigsqcup_u
ut_{(\varpi)}K_{I^{(r)}},
$$
where $u$ runs through a system of representatives of $U_n(\OO_{k,\mathfrak{p}})/t_{(\varpi)}U_n(\OO_{k,\mathfrak{p}})t_{(\varpi)}^{-1}$. Furthermore
$$
V_{\mathfrak{p}}'=
K_{I^{(r)}} \varpi t_{(\varpi)}K_{I^{(r)}}=
\bigsqcup_u
u\varpi t_{(\varpi)}K_{I^{(r)}},
$$
\end{lemma}

\begin{proof}
The first part of the lemma is the same as \citep[Lemma 4.1]{kazhdanmazurschmidt2000}, at least when $0\leq \nu< n$. The case $\nu=n$ as well as the last part follow from the identity
$$
T_n=
V_{\mathfrak{p},n},
$$
which is an immediate consequence of Gritsenko's factorization \eqref{eq:gritsenko}.
\end{proof}

\subsection{The projection formula}

Let $\underline\lambda=(\lambda_1,\dots,\lambda_{m})\in E^{m}$ for $0\leq m\leq n$. We define the $\mathcal H_{I}$-submodule $\mathcal M^{\underline{\lambda}}$ of $\mathcal M$ consisting of all $\psi\in\mathcal M$ such that
\begin{equation}
\forall \nu=1,2,\dots,m:\;\;\;H_\mathfrak{p}(\lambda_\nu)\cdot \psi=0.
\label{eq:heckeroots}
\end{equation}
Furthermore we set
$$
\eta_\nu:=\absNorm(\mathfrak{p})^{-\frac{\nu(\nu-1)}{2}}\prod_{i=1}^\nu\lambda_i
$$
for $1\leq\nu\leq m$. We denote by $\mathcal M_{\underline{\lambda}}$ the $\mathcal H_{I}$-submodule of $\mathcal M^{\underline{\lambda}}$ consisting of vectors $\psi\in\mathcal M$ that are simultaneous eigen functions for $V_{\mathfrak{p},1},\dots,V_{\mathfrak{p},m}$ with eigen value $\eta_\nu$, i.e.\ the subspace of $\psi$ satisfying
$$
V_{\mathfrak{p},\nu}\cdot\psi=\eta_\nu\cdot\psi
$$
for $1\leq\nu\leq m$.

\begin{proposition}\label{prop:heckemodifikation}
Let $\mathcal M$ be a $\mathcal H_{I}$-module over a field $E$. Let $\underline\lambda=(\lambda_1,\dots,\lambda_{m})\in E^{m}$ for $0\leq m\leq n$. Then the map
$$
\Pi_{\underline{\lambda}}^0:\;\psi\;\mapsto\;
\prod_{i=1}^{m}
\prod_{\begin{subarray}cj=1\\j\neq i\end{subarray}}^{n}
(\lambda_i\absNorm(\mathfrak{p})^{1-j}V_{\mathfrak{p},j-1}-V_{\mathfrak{p},j})
\cdot \psi
$$
is a well defined $\mathcal H_{I}$-module map
$$
\Pi_{\underline{\lambda}}^0:\mathcal M^{\underline{\lambda}}\to\mathcal M_{\underline{\lambda}}.
$$
\end{proposition}

\begin{proof}
As the Hecke operators $U_1, \dots, U_n$  commute with $V_{\mathfrak{p},0},\dots,V_{\mathfrak{p},m}$, we see that $\Pi_{\underline{\lambda}}^0$ is indeed an endomorphism of $\mathcal M^{\underline{\lambda}}$.

That $\Pi_{\underline{\lambda}}^0$ is a well defined $\mathcal H_{I}$-module homomorphism $\mathcal M^{\underline{\lambda}}\to \mathcal M_{\underline{\lambda}}$ was proven for $k=\QQ$, $m=n-1$ and $r=1$ in \citep[Proposition 4.2]{kazhdanmazurschmidt2000}. The proof given there eventually shows the slightly more general statement for any $k$, $m$ and $r$.
\end{proof}

\begin{proposition}\label{prop:heckeprojection}
Let $\mathcal M$ be a $\mathcal H_{I}$-module over a field $E$. Let $\underline\lambda=(\lambda_1,\dots,\lambda_{m})\in E^{m}$ with pairwise distinct non-zero $\lambda_1,\dots,\lambda_{m})$ for $0\leq m\leq n$. Then the map
$$
\Pi_{\underline{\lambda}}:\;\psi\;\mapsto\;
\prod_{i=1}^{m}
\prod_{\begin{subarray}cj=1\\j\neq i\end{subarray}}^{n}
\frac{\lambda_i\absNorm(\mathfrak{p})^{1-j}V_{\mathfrak{p},j-1}-V_{\mathfrak{p},j}}
{\lambda_i\cdot\absNorm(\mathfrak{p})^{1-j}\cdot\eta_{j-1}-\eta_j}
\cdot \psi
$$
is a well defined projection
$$
\Pi_{\underline{\lambda}}:\mathcal M^{\underline{\lambda}}\to\mathcal M_{\underline{\lambda}}.
$$
\end{proposition}

\begin{proof}
Due to Proposition \eqref{prop:heckemodifikation} it remains only to show that $\Pi_{\underline{\lambda}}$ is indeed a projection, i.e.\ induces the identity on $\mathcal M_{\underline{\lambda}}$. We proof this by induction on $m$, the case $m=0$ being clear. Assume that $m>0$. Set $\underline{\lambda}':=(\lambda_1,\dots,\lambda_{m-1})$ and by our induction hypothesis $\Pi_{\underline{\lambda}'}$ induces the identity on $\mathcal M^{\underline{\lambda}'}$. Pick any $\psi\in\mathcal M^{\underline{\lambda}}$. Then $\psi$ lies in $\mathcal M^{\underline{\lambda}'}$ and therefore
$$
\Pi_{\underline{\lambda}}(\psi)=
\prod_{\begin{subarray}cj=1\\j\neq m\end{subarray}}^{n}
\frac{\lambda_m\absNorm(\mathfrak{p})^{1-j}V_{\mathfrak{p},j-1}-V_{\mathfrak{p},j}}
{\lambda_m\cdot\absNorm(\mathfrak{p})^{1-j}\cdot\eta_{j-1}-\eta_j}
\cdot
\Pi_{\underline{\lambda}'}(\psi)=
\prod_{\begin{subarray}cj=1\\j\neq m\end{subarray}}^{n}
\frac{\lambda_m\absNorm(\mathfrak{p})^{1-j}V_{\mathfrak{p},j-1}-V_{\mathfrak{p},j}}
{\lambda_m\cdot\absNorm(\mathfrak{p})^{1-j}\cdot\eta_{j-1}-\eta_j}
\cdot\psi=
$$
$$
\prod_{\begin{subarray}cj=1\\j\neq n\end{subarray}}^{n}
\frac{\lambda_m\cdot\absNorm(\mathfrak{p})^{1-j}\cdot\eta_{j-1}-\eta_j}
{\lambda_m\cdot\absNorm(\mathfrak{p})^{1-j}\cdot\eta_{j-1}-\eta_j}
\cdot\psi=\psi,
$$
because $\psi$ is an eigen vector for $V_{\mathfrak{p},}$ with eigen value $\eta_m$.
\end{proof}

\section{The Birch Lemma}\label{sec:birch}

In this section we generalize the Birch Lemma of \citep{januszewski2009} to pairs $(\pi,\sigma)$ which are allowed to be of level $K_{I^{(r)}}$ at $\mathfrak{p}$ and are minimal among $\mathfrak{p}$-power twists. We also renormalize the Birch Lemma, which enables us to overcome the class number restriction encountered in \citep{januszewski2009}.

\subsection{The local Zeta integral}\label{sec:birchlocal}

We use the notation of \cite[Section 2]{januszewski2009} in the following modified setting. Let $\chi:F^\times\to\CC^\times$ be a character of a local field $F$ of non-trivial conductor $\mathfrak{f}_\chi$ generated by $f_\chi=\varpi^s$. We fix another element $f=\varpi^r\in\OO_F$ with $r\geq s$ and write $I_n^{(r)}$ for the Iwahori subgroup of $\GL_n(\OO_F)$ of level $f$.

All quantities that are defined relative to $f$ retain their meaning, i.e.\ the matrices $A_n$, $\tilde{A}_n$, $B_n$, $C_n$, $D_n$, $E_n$, $\phi_n$ are all defined with respect to $f=\varpi^r$, as are the groups $J_{l,n}$ and $\overline{T}_{l,n}$.

We define $\mathfrak{R}_{l,n}$ and its variants as before, i.e.\ via $I_n=I_n^{(1)}$. Assume as before that $l\geq 2n$. We have
$$
J_{l,n}\subseteq I_n^{(r)}\cap w_nD_n^{-1}I_n^{(r)}D_nw_n,
$$
generalizing equation (6) of loc.\ cit.. 

For any $\delta\in\ZZ$ we define
$$
j_\delta:\GL_n(F)\to\GL_{n+1}(F),
$$
$$
g\mapsto
\begin{pmatrix}
g&0\\
0&\varpi^\delta
\end{pmatrix},
$$
and
$$
\lambda_n^\delta(g):=\lambda_n(\varpi^{-\delta}\cdot g).
$$
We need the following generalized statement of Lemma 2.6 of loc.\ cit.
\begin{lemma}\label{lem:birchinductionrelation}
For any $I_{n+1}^{(r)}$- resp.\ $I_n^{(r)}$-invariant $\psi$- resp.\ $\psi^{-1}$-Whittaker functions $w$ and $v$ on $\GL_{n+1}(F)$ resp.\ $\GL_{n}(F)$ and any $\delta\in\ZZ$ we have
$$
w\left(j_\delta(g)C_{n+1}\cdot D_{n+1}w_{n+1}\right) v(g)=
$$
\begin{equation}
\psi\left(\lambda_n^\delta(gB_n)\right) w\left(j_\delta(gB_n\cdot D_nw_n)\right) v(gB_n).
\label{eq:lemma26}
\end{equation}\end{lemma}

\begin{proof}
First observe that all relations in the proof of Lemma 2.6 of loc.\ cit.\ eventually are valid modulo $I_{n+1}^{(r)}$ as well. This shows in particular the case $\delta=0$. The general case may be reduced to this case as follows. We have
$$
j_\delta(g)=\Delta_\delta\cdot j_0(g),
$$
where
$$
\Delta_\delta:=\diag(1,\dots,1,\varpi^\delta).
$$
From the aforementioned proof of Lemma 2.6 in loc.\ cit.\ we know that
$$
w(\Delta_\delta u\Delta_\delta^{-1}j_\delta(g)C_{n+1}D_{n+1}w_{n+1})=
w(j_\delta(gB_{n}D_{n}w_n),
$$
which together with
$$
\Psi(\Delta_\delta u\Delta_\delta^{-1})^{-1}=\Psi(\lambda_n(\varpi^{-\delta}\cdot gB_n))
$$
concludes the proof.
\end{proof}

Set $\delta:=r-s$ and $d_\delta:=(\delta\cdot(n+1-i))_{1\leq i\leq n}\in\ZZ^n$. For any integer $\delta'\in\ZZ$ we write $(\delta')\in\ZZ^n$ for the vector which has $\delta'$ in each components.
\begin{lemma}\label{lem:zentrales}
Let $w$ and $v$ be Iwahori invariant $\psi-$ (resp. $\psi^{-1}$-) Whittaker functions on $\GL_n(F)$. For any $n\geq 0$, $e\in\ZZ^n$, $\omega\in W_n$, $l\geq\max\{2n,n-e_1/\nu_\mathfrak{p}(f),\dots,n-e_n/\nu_\mathfrak{p}(f)\}$ and $\delta'\in\ZZ$ we have
$$
\sum_{g\in\varpi^e\omega \mathfrak{R}_{l,n}^\omega}
\!\!\!\psi(\lambda_n^{\delta'}(g))\cdot
w(
g\cdot D_n w_n
)\cdot
v(g)\cdot
\chi(\det(g))\cdot
\absnorm{\det(g)}^{s}=
$$
$$
\begin{cases}
\absNorm(\mathfrak{f})^{\frac{(l-2n+1)n(n+1)}{2}+\frac{1}{2}\sum_{\nu=1}^{n}5\nu^2-3\nu}
\absNorm(\mathfrak{f}_\chi)^{-\frac{n(n+1)}{2}}
(\chi(f_\chi)G(\chi))^{\frac{n(n+1)}{2}}\\
\;\;\;\;\;\;\;\;\;\;\;\;
\cdot w(\varpi^{\delta+\delta'} t_{(\varpi)}^{\delta}) v(\varpi^{\delta+\delta'} t_{(\varpi)}^{\delta})
\chi(\det \varpi^{\delta+\delta'} t_{(\varpi)}^{\delta})
\absnorm{\det \varpi^{\delta+\delta'} t_{(\varpi)}^{\delta}}^{s},\\
\;\;\;\;\;\;\;\;\;\;\;\;\;\;\;\;\;\;\;\;\;\;\;\;\;\;\;\;\;\;\;\;\;\;\;\;
\;\;\;\;\;\;\;\;\;\;\;\;\;\;\;\;\;\;\;\;\;\;\;\;\;\;\;\;\;\;\;\;\;\;\;\;
\text{for $\omega={\bf1}_n$ and $e=d_{\delta}+(\delta')$,}\\
0,\hfill\text{otherwise}.
\end{cases}
$$
\end{lemma}

\begin{proof}
We closely follow the proof of Lemma 2.7 in loc.\ cit., and only briefly indicate the necessary modifications here. We proceed again by induction, the the partial sums $Z(r)$ being defined as before, using $\lambda_n^{\delta'}$ instead of $\lambda_n$.

In the argument of $v$ the parameter $r$ might not be dropped in our setting in the course of the proof. However, this modification is straightforward and we cease to indicate it again.

We have the generalized relation
$$
\varpi^e\omega {}^\gamma r\cdot D_n w_n\;\in\;
\varpi^e\omega r\cdot D_n  w_n\cdot I_n^{(r)},
$$
which implies that, up to the abovementioned missing $r$ in the argument of the Whittaker function $v$ and the replacement of $\lambda_n$ by $\lambda_n^{\delta'}$, the formula for $Z(r)$ in the bottom of page 20 of loc.\ cit. remains valid, and the second formula on page 21 now reads
$$
\sum_{\gamma\in S}
\chi\left({}^\gamma{\bf1}_n\right)\cdot
\psi(\lambda_n^{\delta'}(\varpi^e\omega{}^\gamma r))=
$$
$$
\prod_{\nu=1}^{n}
\sum_{\gamma_\nu\in\left(\OO_F/\mathfrak{f}^{l}\right)^\times}
\chi(\gamma_\nu)\cdot
\psi\left(\varpi^{e_n-\delta'} f^{\nu-n-1} r_{\sigma(n)\nu}\cdot\gamma_\nu\right).
$$
Therefore the analogue of conclusion (10) of loc.\ cit.\ here is
$$
e_n\neq (n-\sigma(n))\cdot r+\delta+\delta'\;\Rightarrow\;Z(r)=0.
$$
Hence we may assume that
$$
e_n= (n-\sigma(n))\cdot r+\delta+\delta',
$$
which means that if $\sigma(n)\neq n$, then $e_n>\delta+\delta'$. This then implies
$$
\absnorm{\varpi^{e_n-\delta'} f^{n-n-1} r_{\sigma(n)n}\cdot\gamma_{n}}<\absnorm{f_\chi^{-1}},
$$
yielding again
$$
Z(r)=0.
$$
Therefore we can assume $\sigma(n)=n$ and $e_n=\delta+\delta'$. The implication (11) of loc.\ cit.\ is valid without change, and so we may restrict to the case $r_{1n}=f^{n-1}$ and
$$
r_{n\nu}=-f^{n-\nu},\;\;\;2\leq\nu\leq n
$$
as before. Due to our modified Gau\ss{} sum equation (12) of loc.\ cit.\ now reads
\begin{equation}
\sum_{\gamma\in S}
\chi\left({}^\gamma{\bf1}_n\right)\cdot
\psi(\lambda_n^{\delta'}(\varpi^e\omega{}^\gamma r))=
\chi(B_n)\cdot (\chi(f_\chi)G(\chi))^n\cdot\absNorm(\mathfrak{f})^{l\cdot n}\absNorm(\mathfrak{f}_\chi)^{-n}.
\label{eq:gausssummen}
\end{equation}
The rest of the proof remains valid with the obvious changes that are implied by $e_n=\delta+\delta'$ and the distinction between $f_\chi$ and $f$, thanks to the validity of \eqref{eq:lemma26}, which allows for the same inductive argument to remain intact. To make this precise, the relation \eqref{eq:lemma26} and the identity \eqref{eq:gausssummen} together imply that
$$
\sum_{
r\in
\tilde{\mathfrak{R}}_{l,n}^\omega}Z(r)=
(\chi(f_\chi)G(\chi))^n\cdot
\absNorm(\mathfrak{f})^{l\cdot n}\absNorm(\mathfrak{f}_\chi)^{-n}
\absnorm{\varpi^{\delta+\delta'}}^s\chi(\varpi^{\delta+\delta'})\cdot
\sum_{\tilde{r}\in
\mathfrak{R}_{l,n-1}^{\tilde{\omega}}}\tilde{Z}(\tilde{r}),
$$
where the partial sum $\tilde{Z}(\tilde{r})$ is defined mutatis mutandis as $Z(r)$ for the truncated parameters $\tilde{e}$, $\tilde{\omega}$ and $\tilde{r}$ for $\GL_{n-1}(F)$ and the map $\lambda_{n-1}^{\delta+\delta'}$.

The induction hypothesis shows that those partial sums vanish whenever $\omega\neq{\bf1}_n$ or
$$
e\neq d_{\delta}+(\delta+\delta').
$$
Introduce the map
$$
\tilde{j}_{\delta+\delta'}:\GL_{n-1}(F)\to\GL_n(F),
$$
$$
\tilde{g}\mapsto
\begin{pmatrix}
\tilde{g}&0\\
0&\varpi^{\delta+\delta'}
\end{pmatrix},
$$
We get for $e=d_{\delta}+(\delta')$ and $\omega={\bf1}_n$, using the notation $\tilde{d}_\delta$ for the obvious truncation,
$$
\sum_{g\in\varpi^{d_{\delta}+(\delta')} \mathfrak{R}_{l,n}^{{\bf1}_n}}\!\!
\psi(\lambda_n^{\delta'}(g))
w(g\cdot D_n w_n)
v(g)\chi(g)
\absnorm{\det(g)}^{s}=
\absNorm(\mathfrak{f})^{-\frac{n(n-1)}{2}}\cdot\!\!
\sum_{r\in\tilde{\mathfrak{R}}_{l,n}^\omega}Z(r)=
$$
$$
(\chi(f_\chi)G(\chi))^n\cdot
\absNorm(\mathfrak{f})^{l\cdot n}\cdot
\absNorm(\mathfrak{f}_\chi)^{-n}\cdot
\absnorm{\varpi^{\delta+\delta'}}^s\cdot
\chi(\varpi^{\delta+\delta'})\cdot
$$
$$
\absNorm(\mathfrak{f})^{\frac{(n-1)(n-2)}{2}}\cdot\!\!\!\!\!\!\!\!\!\!
\sum_{\tilde{g}\in\varpi^{\tilde{d}_{\delta}+(\delta+\delta')}\mathfrak{R}_{l,n-1}^{{\bf1}_{n-1}}}\!\!\!\!\!\!
\psi(\lambda_{n-1}^{\delta+\delta'}(\tilde{g}))
w(\tilde{j}(\tilde{g}\cdot D_{n-1} w_{n-1}))
v(\tilde{j}(\tilde{g}))\chi(\tilde{g})
\absnorm{\det(\tilde{g})}^{s}=
$$
$$
(\chi(f_\chi)G(\chi))^n\cdot
\absNorm(\mathfrak{f})^{l\cdot n}\cdot
\absNorm(\mathfrak{f}_\chi)^{-n}\cdot
\absnorm{\varpi^{\delta+\delta'}}^s\cdot
\chi(\varpi^{\delta+\delta'})\cdot
\absNorm(\mathfrak{f}_\chi)^{-\frac{n(n-1)}{2}}\cdot
$$
$$
\absNorm(\mathfrak{f})^{\frac{(n-1)(n-2)}{2}
+
\frac{(l-2(n-1)+1)n(n-1)}{2}+\frac{1}{2}\sum_{\nu=1}^{n-1}5\nu^2-3\nu}
\cdot
$$
$$
(\chi(f_\chi)G(\chi))^{\frac{n(n-1)}{2}}
\cdot w(\varpi^{d_{\delta}+(\delta')}) v(\varpi^{d_{\delta}+(\delta')})
\chi(\varpi^{\tilde{d}_{\delta}+(\delta+\delta')})
\absnorm{\det\varpi^{\tilde{d}_{\delta}+(\delta+\delta')}}^{s},
$$
by our induction hypothesis, and the claim follows.
\end{proof}

\begin{theorem}\label{thm:localbirch}
Let $w$ and $v$ be $\psi$- (resp. $\psi^{-1}$-) Whittaker functions on $\GL_{n+1}(F)$ resp.\ $\GL_n(F)$, Iwahori invariant of level $\mathfrak{f}$, and $\chi:F^\times\to\CC^\times$ a character with conductor $1\neq \mathfrak{f}_\chi\mid\mathfrak{f}$. Then
$$
\int\limits_{U_{n}(F)\backslash{}\GL_{n}(F)}
w\left(
j(g)\cdot t_{(ff_\chi^{-1})}
\cdot h^{(f)}
\right)
v(g\cdot ff_\chi^{-1}t_{(ff_\chi^{-1})})
\chi(\det(g))
\absnorm{\det(g)}^{s-\frac{1}{2}}dg=
$$
$$
\prod_{\nu=1}^{n}
\left({1-\absNorm(\mathfrak{p})^{-\nu}}\right)^{-1}
\cdot
\absNorm(\mathfrak{f})^{-\frac{(n+1)n(n-1)}{6}}\cdot
\absNorm(\mathfrak{f}_\chi)^{-\frac{n(n+1)}{2}}\cdot
(\chi(f_\chi)G(\chi))^{\frac{n(n+1)}{2}}\cdot
$$
$$
w(t_{(ff_\chi^{-1})})\cdot
v(ff_\chi^{-1}\cdot t_{(ff_\chi^{-1})}).
$$
\end{theorem}

\begin{proof}
The proof proceeds as the proofs of Theorem 2.1 and Corollary 2.8 of \citep{januszewski2009}, using Lemma \ref{lem:zentrales}, via the substitution $g\mapsto g\cdot\varpi^d$.
\end{proof}

\subsection{The global Zeta integral}\label{sec:birchglobal}

Choose a global field $k$, i.e.\ a finite extension of $\QQ$ or $\FF_p(T)$ and fix an additive character $\psi:k\backslash\Adeles_k\to\CC$ with a local factorization as in \citep[section 3]{januszewski2009}. Let $\pi$ and $\sigma$ be irreducible cuspidal automorphic representations of $\GL_n(\Adeles_k)$ and $\GL_{n-1}(\Adeles_k)$ respectively. Note that $\pi$ and $\sigma$ are always generic \citep{shalika1974}. By $S_\infty$ we denote the set of infinite places of $k$. Let $S$ denote the set of finite places where $\pi$ or $\sigma$ ramifies. Furthermore fix a finite place $\mathfrak{p}$ such that $\pi_\mathfrak{p}$ and $\sigma_\mathfrak{p}$ possess non-zero $I_{n}^{(r)}$ resp.\ $I_{n-1}^{(r)}$-invariant vectors for some fixed $r\geq 0$.

For an overview of the theory of Rankin-Selberg $L$-function $L(s,\pi\times\sigma)$ as developed in \citep{jpss1979a,jpss1979b,jpss1983,jacquetshalika1990s,cogdellpiatetskishapiro1994,cogdellpiatetskishapiro2004} the reader might consult \citep{kazhdanmazurschmidt2000,januszewski2009} for all facts we use.

At any finite place $\mathfrak{q}$ of $k$ we pick a {\em good tensor} $t_\mathfrak{q}^0\in{\mathscr{W}}(\pi_\mathfrak{q},\psi_\mathfrak{q})\otimes{\mathscr{W}}(\sigma_\mathfrak{q},\psi_\mathfrak{q}^{-1})$ in the local Whittaker spaces such that the local Euler factor at $\mathfrak{q}$ is given by the corresponding local zeta integral for $t_\mathfrak{q}^0$, ie.
$$
L(s,\pi_\mathfrak{q}\times\sigma_\mathfrak{q})=\Psi(t_\mathfrak{q}^0,s),
$$
where the right hand side denotes the local Rankin-Selberg zeta integral (or a finite linear combination of those) as in \citep{jpss1983}. We suppose that $t_{\mathfrak{q}}^0=w^0\otimes v^0$ for class-1 $w^0$ and $v^0$ whenever possible (ie. when $\pi$ and $\sigma$ are spherical at $\mathfrak{q}$). By Shintani's explicit formula \citep{shintani1976} this Euler factor is given explicitly by
$$
L(s,\pi_\mathfrak{q}\times\sigma_\mathfrak{q})=\det({\bf 1}_{n(n-1)}-\absNorm(\mathfrak{q})^{-s}A_{\pi_\mathfrak{q}}\otimes A_{\sigma_\mathfrak{q}})^{-1},
$$
for any place $\mathfrak{q}\not\in S\cup S_\infty$, where $A_{\pi_\mathfrak{q}}$ and $A_{\sigma_\mathfrak{q}}$ denote the corresponding Satake parameters.

Now pick any archimedean Whittaker functions (corresponding to $K$-finite vectors) $(w_\mathfrak{q},v_\mathfrak{q})$ for $\mathfrak{q}\in S_\infty$ and form a pair $(w,v)\in\mathscr{W}_0(\pi,\psi)\times\mathscr{W}_0(\sigma,\psi^{-1})$ of global Whittaker functions with factorizations $w=\Otimes\limits_{\mathfrak{q}} w_\mathfrak{q}$, $v=\Otimes\limits_{\mathfrak{q}} v_\mathfrak{q}$. By Fourier transform we have associated automorphic forms $\phi$ on $\GL_{n}(\Adeles_k)$ and $\varphi$ on $\GL_{n-1}(\Adeles_k)$ respectively. For ${\rm Re}(s)\gg 0$ the Euler product
$$
\prod_{\mathfrak{q}}\Psi(w_\mathfrak{q},v_\mathfrak{q},s)=
\int_{\GL_{n-1}(k)\backslash\GL_{n-1}(\Adeles_k)}
\phi
\left(
j(g)
\right)
\varphi(g)\absnorm{\det(g)}^{s-\frac{1}{2}}dg
$$
converges absolutely and has an analytic continuation to $\CC$, as the right hand side is entire. Furthermore we find an entire function $\Omega$, depending only on the Whittaker functions at infinity, such that for the global $L$-function
$$
\Omega(s)\cdot L(s,\pi\times\sigma)=
\prod_{\mathfrak{q}\in S_\infty}\Psi(w_\mathfrak{q}, v_\mathfrak{q},s)\cdot
\prod_{\mathfrak{q}\not\in S_\infty} \Psi(t_\mathfrak{q}^0,s),
$$
for ${\rm Re}(s)\gg 0$. Writing
$$
w_\infty:=\Otimes\limits_{\mathfrak{q}\in S_\infty}w_\mathfrak{q}
$$
and
$$
v_\infty:=\Otimes\limits_{\mathfrak{q}\in S_\infty}v_\mathfrak{q}
$$
we set
$$
(w_\infty\otimes v_\infty)\otimes\Otimes\limits_{\mathfrak{q}} t_\mathfrak{q}^0=\sum_{\iota}w_\iota\otimes v_\iota,
$$
where any $w_\iota\otimes v_\iota$ is a product of pure tensors, we deduce that with the corresponding associated automorphic forms $(\phi_\iota,\varphi_\iota)$ we get
$$
\Omega(w_\infty\otimes v_\infty,1)(s)\cdot L(s,\pi\times\sigma)=
$$
$$
\sum_{\iota}\int_{\GL_{n-1}(k)\backslash\GL_{n-1}(\Adeles_k)}
\phi_\iota
\left(
j(g)
\right)
\varphi_\iota(g)\absnorm{\det(g)}^{s-\frac{1}{2}}dg,
$$
for the entire function $\Omega(w_\infty\otimes v_\infty,1)(s)=:\Omega(s)$, depending on our choices, which is $\CC$-linear in the first argument, the second argument being reserved for a character of $\GL_1(k\otimes_\QQ\RR)$.

In order to study the twisted $L$-function
$$
L(s,(\pi\times\sigma)\otimes\chi):=L(s,(\pi\otimes\chi)\times\sigma)
$$
for a quasi-character $\chi$ with $\mathfrak{p}$-power conductor $\mathfrak{f}_\chi$ we modify the local Whittaker functions at $\mathfrak{p}$ and allow Iwahori invariant pairs only. For this purpose we write $L^{(\mathfrak{p})}(s,(\pi\times\sigma)\otimes\chi)$ for the above $L$-function with the $\mathfrak{p}$-Euler factor removed.  Note that if the pair $(\pi,\sigma)$ is unramified at $\mathfrak{p}$, then this Euler factor is trivial whenever $\chi$ has non-trivial conductor.

\begin{theorem}\label{thm:globalbirch}
For any choice of pair of Whittaker functions $(w_\infty,v_\infty)$ at infinity, and any pair $(w_\mathfrak{p},v_\mathfrak{p})$ of $I_n^{(r)}$- resp. $I_{n-1}^{(r)}$-invariant Whittaker functions on $\GL_{n}(k_\mathfrak{p})$ and $\GL_{n-1}(k_\mathfrak{p})$ respectively, there exists an entire function $\Omega(w_\infty\otimes v_\infty,\chi_\infty)$, only depending on $(w_\infty\otimes v_\infty)$ and $\chi_\infty$, such that for any quasi-character $\chi:k^\times\backslash\Adeles_k^\times\to\CC^\times$ with non-trivial $\mathfrak{p}$-power conductor $\mathfrak{f}_\chi\mid\mathfrak{f}=\mathfrak{p}^r$ we have
$$
\Omega(w_\infty\otimes v_\infty,\chi_\infty)(s)
\delta^{(r)}(w_\mathfrak{p}\otimes v_\mathfrak{p},\chi_\mathfrak{p})
(\chi(f_\chi)G(\chi))^{\frac{n(n-1)}{2}}\cdot
$$
$$
\absNorm(\mathfrak{f}_\chi)^{-\frac{n(n-1)}{2}}\cdot
\absNorm(\mathfrak{f})^{-\frac{n(n-1)(n-2)}{6}}\cdot
L^{(\mathfrak{p})}(s,(\pi\otimes\chi)\times\sigma)=
$$
$$
\sum_{\iota}
\int\limits_{\GL_{n-1}(k)\backslash\GL_{n-1}(\Adeles_k)}
\!\!\!\!\!\!\!\!\!\!\!\!\!\!\!\!
\phi_\iota
\left(
j(g) t_{(ff_\chi^{-1})}
h^{(f)}
\right)
\varphi_\iota(g ff_\chi^{-1}t_{(ff_\chi^{-1})})
\chi(\det(g))
\absnorm{\det(g)}^{s-\frac{1}{2}}
dg,
$$
where
$$
\delta^{(r)}(w_\mathfrak{p}\otimes v_\mathfrak{p},\chi_\mathfrak{p}):=
w_{\mathfrak{p}}(t_{(ff_\chi^{-1})})
\cdot
v_{\mathfrak{p}}(ff_\chi^{-1}\cdot t_{(ff_\chi^{-1})})
\cdot
\prod_{\nu=1}^{n}\left(1-\absNorm(\mathfrak{p})^{-\nu}\right)^{-1}.
$$
\end{theorem}

\begin{proof}
It is clear that to compute the twisted $L$-function we might choose at any infinite place $\mathfrak{q}$ the pair of local Whittaker functions $(\chi_\mathfrak{q}(\det)\cdot w_\mathfrak{q}, v_\mathfrak{q})$ for even $n$ or $(w_\mathfrak{q}, \chi_\mathfrak{q}(\det)\cdot v_\mathfrak{q})$ for odd $n$. This data will account for $\Omega(w_\infty\otimes v_\infty,\chi_\infty)$. The rest of the argument is reduced to Theorem \ref{thm:localbirch} as in \citep[Proof of Theorem 3.1]{januszewski2009}, which in turn is a variant of the standard argument in the proof of the Global Birch Lemma in \cite{kazhdanmazurschmidt2000}.
\end{proof}

Let $U_\mathfrak{q}:=\Gm(\OO_{k,\mathfrak{q}})$ for nonarchimedean $\mathfrak{q}$ and define $U_\mathfrak{q}:=\Gm(k_\mathfrak{q})^0$ for $\mathfrak{q}\in S_\infty$. For an id\`ele $\alpha\in\Adeles_k^\times$ we let $C_{\mathfrak{f}}$ denote the preimage of
$$
k^\times\backslash{}k^\times\cdot(1+\mathfrak{f})\cdot\prod_{\mathfrak{q}\nmid\mathfrak{f}} U_\mathfrak{q}
$$
under the determinant map
$$
\det: \GL_{n}(k)\backslash{}\GL_{n}(\Adeles_k)\to k^\times\backslash{}\Adeles_k^\times.
$$
For any id\`ele $x\in\Adeles_k^\times$ we set
$$
d_{(x)}:=\diag(x,1,\dots,1).
$$
As a consequence of Theorem \ref{thm:globalbirch} we have
\begin{corollary}\label{kor:globalbirch}
For any $\chi$ of finite order and conductor $\mathfrak{f}_\chi\mid\mathfrak{f}$ and any $\nu\in\ZZ$ we have
$$
\Omega(w_\infty\otimes v_\infty,\chi_\infty)(\frac{1}{2}+\nu)
\delta^{(r)}(w_\mathfrak{p}\otimes v_\mathfrak{p},\chi_\mathfrak{p})
(\chi(f_\chi)G(\chi))^{\frac{n(n-1)}{2}}\cdot
$$
$$
\absNorm(\mathfrak{f}_\chi)^{-\frac{n(n-1)}{2}}
\absNorm(\mathfrak{f})^{-\frac{n(n-1)(n-2)}{6}}\cdot
L^{(\mathfrak{p})}(\frac{1}{2}+\nu,(\pi\otimes\chi)\times\sigma)=
$$
$$
\sum_{\iota,x}
\chi(x)\cdot
\int_{C_{\mathfrak{f}}}
\phi_\iota
\left(
j(gd_{(x)})\cdot t_{(ff_\chi^{-1})})
\cdot
h^{(f)}
\right)
\cdot
\varphi_\iota(gd_{(x)}\cdot ff_\chi^{-1}t_{(ff_\chi^{-1})})\absnorm{\det(gd_{(x)})}^\nu
dg.
$$
Here $x$ runs through a system of representatives of the ray class group $k^\times\backslash\Adeles_k^\times/(1+\mathfrak{f})\prod_{\mathfrak{q}\neq\mathfrak{p}} U_\mathfrak{q}$.
\end{corollary}

\section{Arithmetic groups and relative Lie algebra cohomology}\label{sec:liecoh}

The historically inclined reader might consult the fundamental articles of Matsushima and Murakami \citep{matsushimamurakami1963,matsushimamurakami1965}. The modern main reference is of course \citep{book_borelwallach1980}. We assume here $k$ to denote a number field. We write $G_n$ for the restricition of scalars of $\GL_n$ in the extension $k/\QQ$. Let $K$ denote a maximal compact subgroup of $G:=G_n(\RR)$ and $\theta$ the corresponding (algebraic) Cartan involution. We write $\lieg=\liegl_n\otimes\CC$ for the complexified Lie algebra of $G$ and $\liek$ for the complexified Lie algebra of $K$. We can identify $\liek$ with the $(+1)$-eigen space of $\theta$ acting on $\lieg$, and likewise we have a $(-1)$-eigen space that we denote $\liep$. Then
$$
\lieg=\liep\oplus\liek,
$$
i.e.\ $\liep$ is canonically identified with $\lieg/\liek$.

\subsection{Relative Lie algebra cohomology}
Pick a $(\lieg,K)$-module $(\pi,V)$, for example the space of $K$-finite vectors in an automorphic representation of $G_n$. By the very definition of $(\lieg,K)$-modules, the compatibility of the actions of $\lieg$ and $K$ on $V$ reads
$$
\pi(z)\cdot\pi(g)\cdot v=\pi(\Ad(z)(g))\cdot\pi(z)\cdot v
$$
for all $g\in\lieg$, $z\in K$. Furthermore, as $v$ is contained in a finite dimensional $K$-stable subspace $W\subseteq V$, we have
$$
\pi(L(z))\cdot w=L(\pi_W(z))\cdot w
$$
for all $w\in W$, where $\pi_W$ denotes the representation of $K$ on $W$ induced by $\pi$. In other words the differential of $\pi_W$ is given by $\pi_\liek$.

Consider the complex
$$
C^q(\lieg,\liek; V):=\Hom_\liek(\bigwedge^q\liep,V),
$$
with the differential $d:C^q\to C^{q+1}$ given by
$$
df(x_0\wedge\dots\wedge x_q)=
\sum_{i} (-1)^i\cdot x_i\cdot f(x_0\wedge\dots\wedge\hat{x}_i\wedge\dots\wedge x_q)+
$$
$$
\sum_{i<j}(-1)^{i+j} f([x_i,x_j]\wedge x_0\wedge\dots\wedge\hat{x}_i\wedge\dots\wedge\hat{x}_j\wedge\dots\wedge x_q).
$$
This gives rise to classical relative Lie algebra cohomology denoted $H^q(\lieg,\liek; V)$. If $V$ is admissible, then this cohomology is finite dimensional, as the complex itself is finite dimensional.

\subsection{$(\lieg,K)$-cohomology}

Now $K$ naturally acts on $\liep$ by the adjoint action and it acts naturally on $V$ as well. So we have another complex
$$
C^q(\lieg, K; V):=\Hom_K(\bigwedge^q\liep,V),
$$
which eventually turns out to be a subcomplex of the former. To identify this subcomplex, note that $\pi_0(G)$ can be canonically identified with $K/K^0$ and the latter group acts naturally on the first complex (again via the adjoint representation on $\liep$). Denote by
$$
i:C^q(\lieg,K; V)\to C^q(\lieg,\liek; V)
$$
the canonical inclusion.

Then for any $f$ in $C^q(\lieg,\liek; V)$ an any $x\in\bigwedge^q \liep$ we know that $f(x)$ is contained in a finite dimensional $K$-stable subspace of $V$. Furthermore $K^0$ acts trivially on $f$ and moreover
$$
C^q(\lieg,\liek; V)=C^q(\lieg,K^0; V):=\Hom_{K^0}(\bigwedge^q\liep, V).
$$
So we eventually get a well defined action of $\pi_0(G)$ on this space. We conclude that for this action
$$
C^q(\lieg,K; V)=\Hom_K(\bigwedge^q\liep, V)=C^q(\lieg,\liek; V)^{\pi_0(G)}.
$$
Finally for the respective cohomologies we get
$$
H^q(\lieg,K; V)=H^q(\lieg,\liek; V)^{\pi_0(G)},
$$
as taking invariants of a semisimple group action is plainly exact.

\subsection{de Rham isomorphism}

Now suppose that $V^{\rm smooth}$ is a smooth admissible representation of $G$, which contains $V$ as $K$-finite vectors. Denote by $\Omega^q(G/K; V^{\rm smooth})$ the space of smooth $V^{\rm smooth}$-valued differential $q$-forms, i.e.\ the space of smooth sections $\omega:G/K\to \bigwedge^q T^*(G/K)\otimes V^{\rm smooth}$. Then $\Omega^q(G/K; V^{\rm smooth})$ becomes a complex with the natural exterior differential. Furthermore we have a natural action of $G$ given by
$$
(g\cdot\omega)_x(X):=g\cdot\omega_{g^{-1}x}(g^{-1}X),
$$
for any $x\in G/K$ and any $X\in \bigwedge^q T_x(G/K)$.

As $G$ acts transitively on $G/K$, we have a natural isomorphism
$$
\Omega^q(G/K; V^{\rm smooth})^G\cong C^q(\lieg,K; V),
$$
given by evaluation
$$
\omega\mapsto\omega_e,
$$
by virtue of the identification $T_e(G/K)=\liep$. Indeed, due to the admissibility the canonical inclusion induces for the $K$-invariant origin $e\in G/K$ an isomorphism of stalks
$$
\left(\bigwedge^q T^*_e(G/K)\otimes V\right)^K\cong
\left(\bigwedge^q T^*_e(G/K)\otimes V^{\rm smooth}\right)^K.
$$

As the differentials of these two complexes are compatible, we get a natural isomorphism
$$
H^q(\Omega^{\bullet}(G/K; V^{\rm Smooth})^G)\cong
H^q(\lieg, K; V)
$$
in cohomology. A similar statement holds for $(\lieg,\liek)$-cohomology. The same reasoning yields canonical isomorphisms
$$
H^q(\Omega^{\bullet}(G/K; V^{\rm Smooth}))\cong
H^q(\Omega^{\bullet}(G/K; V^{\rm Smooth})^{G^0})\cong
H^q(\lieg, \liek; V),
$$
the first isomorphism being classically due to de Rham.

\subsection{Component action}

We have seen that $H^q(\lieg, K; V)$ is the subspace of $H^q(\lieg, \liek; V)$ where $\pi_0(K)=\pi_0(G)$ acts trivially. In our applications it turns out that we also need to consider nontrivial eigen spaces of this action.

Let $\varepsilon$ be a character of $\pi_0(K)$, that we consider also as a character of $K$ via the projection $K\to\pi_0(K)$. Then we have a corresponding eigen space
$$
H^q(\lieg,\liek; V)_\varepsilon:=\{h\in H^q(\lieg,\liek; V)\mid \forall k_0\in\pi_0(K):k_0h = \varepsilon(k_0)\cdot h\}.
$$
Obviously for the trivial character we get
$$
H^q(\lieg,\liek; V)_{\bf 1}= H^q(\lieg,K; V).
$$
We consider $\varepsilon$ as the one-dimensional $(\lieg,K)$-module, on which $(\lieg,K^0)$ acts trivally and on which $\pi_0(K)$ acts via $\varepsilon$. This corresponds to the pullback of $\varepsilon$ along $G\to\pi_0(K)$.

In our application $V$ comes from the infinity component of an irreducible cuspidal automorphic representation and furthermore has non-trivial $(\lieg,K^0)$-cohomology. Then two cases will arise. The first one concerns even $n$. In this case $V$ will be isomorphic to $V\otimes\varepsilon$ for all $\varepsilon$ and we have non-vanishing eigen spaces in cohomology, all of the same dimensions for any $\varepsilon$. For odd $n$ this is not the case and there is a unique choice of $\varepsilon$ such that the eigen space in cohomology does {\em not} vanish.

\subsection{Cohomology of arithmetic groups}

We keep the above notation and let $\rho:G\to E$ denote a smooth finite-dimensional representation of $G$. We can consider the smooth $G$-module $V^{\rm smooth}\otimes E$. Fix an arithmetic subgroup $\Gamma$ of $G$ contained in $G^0$. Then $\rho$ induces a finite dimensional representation of $\Gamma$, that we also denote $E$. It is classical that we have canonically
$$
H^q(\Gamma; E)\cong H^q(\Omega^\bullet(G/K; E)^\Gamma).
$$
Assume for simplicity that $\Gamma$ acts freely. The general case can be deduced from this case via the Hochschild-Serre spectral sequence, as this implies in particular that for a (torsion-free) normal subgroup $\Gamma'\subseteq\Gamma$ of finite index there is a natural isomorphism
$$
H^q(\Gamma; E)=H^q(\Gamma'; E)^{\Gamma/\Gamma'},
$$
and we have a similar statement for the complex of smooth $E$-valued differential forms.

We know that $G/K$ is an Euclidean space, in particular it is contractible. As $\Gamma$ is torsion-free, it acts freely on this space and $\Gamma\backslash G/K$ is an Eilenberg-MacLane space $K(\Gamma,1)$. It results that
$$
H^q(\Gamma; E)\cong H^q(\Gamma\backslash G/K; \underline{E}),
$$
where $\underline{E}$ is the local system associated to $E$, that we consider as a sheaf on $\Gamma\backslash G/K$. For any open $U\subseteq \Gamma\backslash G/K$ the sections are given by locally constant functions that are invariant under $\Gamma$, i.e.
$$
\underline{E}(U)=\{\phi:\Gamma U\to E\mid \forall x\in \Gamma U: f(x)=\rho(\gamma)(f(\gamma^{-1}x))\},
$$
where $\Gamma U$ denotes the preimage of $U$ under the canonical projection $\pi:G/K\to\Gamma\backslash G/K$. Note that this sheaf cohomology can be calculated by our de Rham complex, as the restriction maps are locally constant. For the latter the pullback
$$
\pi^*:\Omega^q(\Gamma\backslash G/K; \underline{E})\to\Omega^q(G/K; E),
$$
$$
\omega\;\mapsto\;\omega\circ\pi
$$
along $\pi$, where on the right hand side we consider $E$ as a constant sheaf, induces an isomorphism
$$
\Omega^q(\Gamma\backslash G/K; \underline{E})\cong\Omega^q(G/K; E)^\Gamma.
$$
We conclude that
$$
H^q(\Gamma; E)\cong 
H^q(\Omega^\bullet(G/K; E)^\Gamma).
$$

On the other hand translation by $g^{-1}\in G$ yields a natural identification of tangent spaces
$$
T_g(G)\to T_e(G)=\lieg.
$$
Therefore we can canonically identify
$$
\Omega^q(G; E)=\Hom(\bigwedge^q\lieg,\mathscr C^\infty(G; E)).
$$
Now $\Gamma$ acts on these spaces in a compatible manner by virtue of a trivial action on the tangent spaces. Then
$$
\omega\in\Omega^q(G; E)^\Gamma
$$
if and only of for any $\gamma\in\Gamma$, $x\in G$, $X\in\bigwedge^q T_x(G)=\lieg$,
$$
\omega_x(X)=\rho(\gamma)(\omega_{\gamma^{-1}x}(X)).
$$
In particular the above identification yields a canonical isomorphism
$$
\Omega^q(G; E)^\Gamma=\Hom(\bigwedge^q\lieg,\mathscr I_\Gamma^\infty(G; E)),
$$
where
$$
\mathscr I_\Gamma^\infty(G; E):=\{\phi\in\mathscr C^\infty(G; E)\mid\forall\gamma\in\Gamma,x\in G: \phi(x)=\rho(\gamma)(\phi(\gamma^{-1}x))\}.
$$
On this space $G$ acts by right translation, i.e.\ this gives the representation smoothly induced from the restriction of $\rho$ to $\Gamma$. Our identifications are compatible with the differentials of our complexes, such that we get
$$
H^q(\Omega^\bullet(G; E)^\Gamma)\cong H^q(\lieg;\mathscr I_\Gamma^\infty(G; E)),
$$
whence our first step towards identifying the cohomology of the arithmetic group with relative Lie algebra cohomology. Now the map
$$
\phi\mapsto \phi^0:g\mapsto\rho(g)^{-1}(\phi(g))
$$
gives an identification
$$
\mathscr I_\Gamma^\infty(G, E)\cong
\mathscr C^\infty(\Gamma\backslash G; E)=
\mathscr C^\infty(\Gamma\backslash G)\otimes E,
$$
of $G$-modules, where $G$ acts on the right hand side by right translation and $\rho$ respectively. So we get an isormophism
$$
H^q(\Gamma\backslash G;\underline{E})\cong
H^q(\Omega^\bullet(G; E)^\Gamma)\cong H^q(\lieg;\mathscr I_\Gamma^\infty(G; E))
\cong H^q(\lieg;\mathscr C^\infty(\Gamma\backslash G)\otimes E).
$$

Eventually the same procedure works for symmetric spaces and relative Lie algebra cohomology. To be more precise, write $\pi:G\to G/K$ for the canonical projection. Then pullback along $\pi$ induces an isomorphism of complexes
$$
\pi^*:\Omega^q(G/K; E)^\Gamma\cong C^q(\lieg, K; \mathscr I_\Gamma^\infty(G; E))\cong C^q(\lieg, \liek; \mathscr I_\Gamma^\infty(G^0; E)),
$$
which in turn induces an isomorphism in cohomology:
$$
H^q(\Gamma; E)\cong
H^q(\Omega^\bullet(G/K; E)^\Gamma)\cong H^q(\lieg, \liek; \mathscr I_\Gamma^\infty(G^0;E))\cong
$$
$$
H^q(\lieg, \liek; \mathscr C^\infty(\Gamma\backslash G^0)\otimes E)\cong
H^q(\lieg, K; \mathscr C^\infty(\Gamma\backslash G)\otimes E).
$$
This construction can be exploited mutatis mutandis with growth conditions and compact support, i.e.\ we always have an isomorphism
$$
H_*^q(\Gamma\backslash G/K; \underline{E})\cong H^q(\lieg, K; \mathscr C_*^\infty(\Gamma\backslash G)\otimes E),
$$
where $*\in\{{\rm c},{\rm cusp}, {\rm fd},{\rm mg}\}$. The beauty of this isomorphism is that the right hand side can be computed purely algebraically, because it does not change when restricting to $K$-finite vectors. Finally we note that once a complex structure on the symmetric space is available, this result may be refined and shown to naturally respect the respective Hodge decompositions.

In suitable situations it can even be guaranteed that
$$
H^q(\lieg, K^0; V\otimes E)=(\bigwedge^q\liep^*\otimes V\otimes E)^{K^0},
$$
which is the case when $V^{\rm smooth}$ is unitary and $E$ an algebraic representation. The reason being that, as is well known, the Casimir operators of the representations act by scalars, so we might apply \citep[Chapter II, Proposition 3.1]{book_borelwallach1980} to the connected component, which yields the result for $(\lieg,\liek)$-cohomology. Taking $K$-invariants we might recover $(\lieg,K)$-cohomology. This is a vast generalization of a classical result of E. Cartan saying that on $G/K$ all $G^0$-invariant forms are harmonic, closed and co-closed.

We pick a closed connected $\theta$-stable subgroup $S$ in the center of $G$, with Lie algebra $\lies$. Then we consider the manifold $G/KS$ whose tangent space at the identity is given by
$$
\lieg/(\liek+\lies)=\liep/(\liep\cap\lies).
$$
Obviously we might assume $\lies\subseteq\liep$, which is the same as to say that all compact subgroups in $S$ are trivial. We assume that $S\cap\Gamma$ is trivial. This is in particular the case if $S$ is the connected component of the $\RR$-valued points of a $\QQ$-split torus in the center of $G_n$. Then all of the above results hold with the following modifications. We define the $(\lieg,KS)$-cohomology as the cohomology of the complex
$$
C^q(\lieg,K; V)^S=C^q(\lieg,\liek+\lies; V)^{\pi_0(G)},
$$
for which we write $H^q(\lieg,KS; V)$. Then
$$
H^q(\Gamma; E)\cong
H^q(\Omega^\bullet(G/KS; E)^\Gamma)\cong H^q(\lieg, K; \mathscr I_\Gamma^\infty(G/S;E))\cong
$$
$$
H^q(\lieg/\lies, KS; \mathscr C^\infty(\Gamma\backslash G/S)\otimes E),
$$
by the very same argument.

\section{Cohomological construction of the distribution}\label{sec:distribution}

In this section we suppose that $k$ is totally real and that $\pi$ and $\sigma$ are {\em regular} and {\em algebraic} in the sense of \citep[Definition 1.8, Definition 3.12]{clozel1990}. We use the modular symbols constructed in \citep{januszewski2009} with possibly nontrivial coefficients to deduce that the distribution above is in fact algebraic and $\mathfrak{p}$-adically bounded, i.e.\ a $\mathfrak{p}$-adic measure. For the mere treatment of algebraicity over $\QQ$ see \citep{kastenschmidt2008} and for the case of trivial coefficients over an arbitrary number field consult \citep{januszewski2009}.

\subsection{Cohomological interpretation of the period integrals}\label{sec:cohint}

We identify $S_\infty$ with the set of embeddings $k\to\RR$. We fix the standard torus $T_n$ in $\GL_n$ and consider all root data for $\GL_n$ resp.\ $G_n$ with respect to $T_n$ resp.\ $\res_{k/\QQ}T_n$ and the ordering induced by the standard Borel subgroup $B_n$ resp.\ $\res_{k/\QQ}B_n$. We make the usual standard choice of simple roots, and we choose as basis of characters the component projections $\chi_j:T_n\to\Gm$, $(t_i)\mapsto t_j$. Furthermore for any $\iota\in S_\infty$ we are given an irreducible representation $\rho_{{\boldsymbol{\mu}}_\iota}$ of $\GL_n$ of heighest weight ${\boldsymbol{\mu}}_\iota=({\boldsymbol{\mu}}_{\iota,i})_{1\leq i\leq n}\in\ZZ^n=X(T_n)$, which is supposed to be dominant and regular. Likewise we have an irreducible representation of highest weight ${\boldsymbol{\nu}}_\iota$ of $\GL_{n-1}$, which is defined over a number field $\QQ(\boldsymbol{\mu})$, the field of rationality of $\boldsymbol{\mu}$ (the Galois action on $\boldsymbol{\mu}$ is induced by the Galois action on the embeddings $\iota\in S_\infty$). For $\boldsymbol{\mu}=(\boldsymbol{\mu_\iota})_{\iota\in S_\infty}$ we write $M_{\boldsymbol{\mu}}$ for a fixed $\OO_{\QQ(\boldsymbol{\mu})}$-model of the representation space of $\rho_{\boldsymbol{\mu}}=\otimes_{\iota\in S_\infty}\rho_{{\boldsymbol{\mu}}_\iota}$, i.e.\ an $\OO_{\QQ(\boldsymbol{\mu})}$-scheme in modules. More concretely we assume a flat $\OO_{\QQ(\boldsymbol{\mu})})$-module
$$
M_{{\boldsymbol{\mu}}}(\OO_{\QQ(\boldsymbol{\mu})})
$$
to be given such that for any $\OO_{\QQ(\boldsymbol{\mu})}$-algebra $A$ we have the $A$-valued points
$$
M_{{\boldsymbol{\mu}}}(A)=M_{{\boldsymbol{\mu}}}(\OO_{\QQ(\boldsymbol{\mu})})\otimes_{\OO_{\QQ(\boldsymbol{\mu})}} A.
$$
The flatness guarantees that the natural map
$$
M_{{\boldsymbol{\mu}}}(\OO_{\QQ(\boldsymbol{\mu})})\to
M_{{\boldsymbol{\mu}}}(\QQ(\boldsymbol{\mu}))
$$
is a monomorphism. Then we consider $M_{{\boldsymbol{\mu}}}$ to be the corresponding irreducible representation of $G_n=\res_{\OO_k/\ZZ}\GL_n$, the latter canonically identified with $\GL_n^{S_\infty}$ after sufficient extension of scalars. We have a well defined diagonal embedding $\GL_n\to G_n$ over $\ZZ$. In particular we write $w_n$ for the long Weyl element in $\GL_n$, embedded diagonally into $G_n$. It should be clear from the context when we are talking about a diagonally embedded $w_n$ or not.

To clarify the relation between finite dimensional representations and the critical values we need to introduce some more notation and terminology. See \citep[section 1]{januszewski2009} for the details concerning the general setup for reductive groups. Write $G_n:=\res_{\OO_k/\ZZ}\GL_n$ and introduce the group
$$
{}^0G_n:=\bigcap_{\alpha\in X_\QQ(G)}\kernel\alpha^2.
$$
Then ${}^0G_n$ is a reductive group scheme over $\ZZ$ and $X_\QQ({}^0G_n)\otimes_\ZZ\QQ=1$ \citep[Proposition 1.2]{januszewski2009}. Furthermore denote by $S$ the maximal $\QQ$-split torus in the radical of $G_n$, or in the maximal $\QQ$-split central torus of $G_n$, what amounts to the same. Then
$$
G_n(\RR)={}^0G_n(\RR)\rtimes S(\RR)^0,
$$
cf. \citep[Proposition 1.2]{borelserre1973}. In our case $S(\RR)^0$ is isomorphic to $\Gm(\RR)^0$ as a (real) Lie group and furthermore
$$
\rang_\RR\zentrum(G_n)=[k:\QQ]
$$
as $k$ is totally real, cf. \citep[section 5]{januszewski2009}. We have explicitly
$$
{}^0G_n(\RR)=\{g\in G_n(\RR)\mid \prod_{\iota\in S_\infty}\absnorm{\det g_\iota}_\iota=1\},
$$
furthermore we introduce the subgroup
$$
G_n^\pm=\{g\in G_n(\RR)\mid \forall \iota\in S_\infty:\absnorm{\det g_\iota}_\iota=1\}.
$$
The reason for considering the latter lies in the fact that, contrary to the former, it has a natural decomposition into local components, which allows to apply the K\"unneth formalism to study its Lie algebra cohomology, which in turn is related to the de Rham cohomology of the symmetric spaces associated to the former or equivalently to the Lie algebra cohomology of that group.

We suppose that for all $\iota\in S_\infty$
$$
H^\bullet(\liegl_{n,\iota},K_{n,\iota}; \pi_\iota\otimes M_{{\boldsymbol{\mu}}_\iota}(\CC))\neq 0,
$$
and likewise
$$
H^\bullet(\liegl_{n-1,\iota},K_{n-1,\iota}; \sigma_\iota\otimes M_{{\boldsymbol{\nu}}_\iota}(\CC))\neq 0,
$$
where $\liegl_{n,\iota}$ is the komplexified Lie algebra of $\GL_n(k_\iota)$ and $K_{n,\iota}:=\iota(O(n)Z_n^0)$ is the product of the connected component of the center of $\GL_n(k_\iota)$ and its standard maximal compact subgroup. Now regularity of $\pi$ means that the ${\boldsymbol{\mu}}_{\iota,i}$ are pairwise distinct, i.e.\ ${\boldsymbol{\mu}}_{\iota,i}>{\boldsymbol{\mu}}_{\iota,i+1}$ for all $1\leq i <n$. Furthermore we know that the {\em weight}
$$
{\bf w}:={\boldsymbol{\mu}}_{\iota,i}+{\boldsymbol{\mu}}_{\iota,n+1-i}
$$
is {\em independent} of $i$ and $\iota\in S_\infty$, and similarly
$$
{\bf v}:={\boldsymbol{\nu}}_{\iota,i}+{\boldsymbol{\nu}}_{\iota,n-i}.
$$
In other words, $\pi_\infty$ and $\sigma_\infty$ are selfcontragredient up to twist.

We write for any $\mu\in\ZZ$
$$
M_{\boldsymbol{\mu}}[\mu]:=M_{\boldsymbol{\nu}}\otimes M_{(\mu)},
$$
where $M_{(\mu)}$ denotes the one-dimensional $G_n$-module on which $G_n$ acts via the projection $G_n\to G_n/({}^0G_n)^0\cong\Gm$, and $\Gm$ acts via the character $x\mapsto x^\mu$. Assume that our integral models are chosen in such a way that we can fix for all dominant weights $\boldsymbol{\mu}$ and all $\mu\in\ZZ$ isomorphisms
\begin{equation}
M_{\boldsymbol{\nu}}[\nu]\to M_{{\boldsymbol{\nu}}+\nu},
\label{eq:mnutwistiso}
\end{equation}
where $\boldsymbol{\mu}+\mu$ is considered to be the collection of weights $(\boldsymbol{\mu}_{\iota,i}+\mu)_{1\leq i \leq n}$.

We write $\check{{\boldsymbol{\mu}}}$ for the highest weight of the contragredient representation corresponding to ${\boldsymbol{\mu}}$, and assume that we have another twisted isomorphism
$$
\cdot^\vee:M_{\boldsymbol{\mu}}\to M_{\check{\boldsymbol{\mu}}},
$$
which induces an isomorphism if we twist the action on $M_{\boldsymbol{\mu}}$ by the twisted main involution $g\mapsto w_{n}(g^{-1})^t w_{n}$, for the long Weyl element $w_n\in\GL_n$. Furthemore we assume that $\cdot^\vee$ is compatible with the above isomorphisms.

By a straightforward computation (cf.\ \citep[Proposition 2.2]{kastenschmidt2008}) we know that there exists a central critical half-integer $s=\frac{1}{2}+\nu$ for $L(s,\pi\times\sigma)$ if and only if
\begin{equation}
{\bf w}\equiv {\bf v}\pmod{2}
\label{eq:parity}
\end{equation}
and all critical half-integers are centered around
$$
s=\frac{1+{\bf w}+{\bf v}}{2},
$$
which is itself critical. Furthermore under \eqref{eq:parity} this condition of criticality is equivalent to the existence of a $G_{n-1}$-equivariant embedding $M_{\check{\boldsymbol{\nu}}}[\nu]\to M_{\boldsymbol{\mu}}$. Then the map
$$
\nu\mapsto s=\frac{1}{2}+\nu
$$
sets up a bijection between the set ${\rm Emb}(\check{{\boldsymbol{\nu}}},{\boldsymbol{\mu}})$ of integers $\nu$ such that there is an embedding $M_{\check{\boldsymbol{\nu}}}[\nu]\to M_{\boldsymbol{\mu}}$ of $G_{n-1}$-modules and the set of critical half-integers for $L(s,\pi\times\sigma)$.

Denote $(\boldsymbol{w},\boldsymbol{l})$ the Langlands parameter for $\pi_\infty$, i.e.
$$
\boldsymbol{l}=2\cdot(\boldsymbol{\mu}+\rho_n)-(\boldsymbol{w}),
$$
where $\rho_n$ denotes the half sum of the positive roots for our choice of root datum in $G_n$ and $(\boldsymbol{w})$ denotes the diagonally embedded weight $\boldsymbol{w}$. Then $\boldsymbol{l}=(l_{\iota,i})_{\iota\in S_\infty,1\leq i\leq n}$ and we use a similar notation for the Langlands parameter $(\boldsymbol{v},\boldsymbol{m})$ for $\sigma_\infty$. With this notation we set
$$
\nu_{\min}:=\frac{\boldsymbol{w}+\boldsymbol{v}}{2}-\min_{i,j,\iota}|l_{\iota,i}-m_{\iota,j}|+1.
$$
Note that due to our parity condition \eqref{eq:parity}, we know that
$$
\nu_{\min}\leq \frac{\boldsymbol{w}+\boldsymbol{v}}{2}.
$$
With this notation
$$
s_{\min}:=\frac{1}{2}+\nu_{\min}
$$
is the left most critical value for $L(s,\pi\otimes\sigma)$, and the right most is
$$
s_{\max}:=\frac{1}{2}+\boldsymbol{w}+\boldsymbol{v}-\nu_{\min}.
$$

Write $\mathscr X_n$, ${}^0\mathscr X_n$, $\mathscr X_n^{\ad}$ for the symmetric spaces associated to $G_n(\RR)$, ${}^0G_n(\RR)$ and $G_n(\RR)/S(\RR)^0$ respectively. We use the same supscripts for arithmetic subgroups and their various projections and restrictions. Note however that in this sense ${}^0\Gamma=\Gamma$ for any arithmetic $\Gamma\subseteq G_n(\RR)$, so that we drop the supscript \lq$0$\rq{} from the notation. Considering arithmetic quotients we also tend to drop the supscript \lq${\ad}$\rq{}.

Consider the composition
$$
s:=\ad\circ j:G_{n-1}\to G_n\to G_n^{\ad}
$$
of group schemes over $\ZZ$. This is eventually a central morphism which can be interpreted in cohomology with suitable growths conditions by means of \citep[Proposition 1.4]{januszewski2009}. For any $\gamma\in G_n(\QQ)$ and sufficiently compatible arithmetic subgroups $\Gamma_{n-1}\subseteq G_{n-1}(\QQ)$, $\Gamma_{n}\subseteq G_{n}(\QQ)$, i.e.\ $\gamma^{-1} s(\Gamma_{n-1})\gamma\subseteq\Gamma_{n}$, we have a natural map in cohomology
$$
s_{\gamma}^*:H_{\rm c}^q(\Gamma_n\backslash\mathscr X_n^{\ad})\to
H_{\rm c}^{q}(\Gamma_{n-1}\backslash\mathscr X_{n-1}),
$$
which is induced by the pullback $s_\gamma^*$ along the map
$$
s_\gamma:\Gamma_{n-1}\backslash\mathscr X_{n-1}\to\Gamma_n\backslash\mathscr X_n^{\ad},
$$
$$
\Gamma_{n-1}x\mapsto\Gamma_n\gamma^{-1} s(x),
$$
which is proper in fact.

Now when considering $M_{{\boldsymbol{\mu}}}$ as a $\GL_{n-1}$-module that we denote $j^*(M_{{\boldsymbol{\mu}}})$, it obviously decomposes into a finite sum of irreducible modules
$$
j^*(M_{{\boldsymbol{\mu}}})=\bigoplus_{\boldsymbol{\mu}^*}M_{\boldsymbol{\mu}^*},
$$
where $\boldsymbol{\mu}^*$ is a highest weight for $G_{n-1}$. By the classical work of Weyl \citep{book_weyl1939} the sum ranges over all such $\boldsymbol{\mu}^*$ satisfying for all $\iota\in S_\infty$ and all $1\leq i < n$
$$
\boldsymbol{\mu}_{\iota,i}\geq
\boldsymbol{\mu}^*_{\iota,i}\geq
\boldsymbol{\mu}_{\iota,i+1}.
$$
We claim that
$$
\Hom_{\Gamma_{n-1}}(M_{\check{\boldsymbol{\nu}}},j^*(M_{\boldsymbol{\mu}}))
$$
has a natural basis enumerated by ${\rm Emb}(\check{\boldsymbol{\nu}},\boldsymbol{\mu})$.

Indeed, Borel has shown that $\Gamma_{n-1}^{\der}:=\Gamma_{n-1}\cap G_{n-1}^{\der}(\RR)$ is Zariski dense in $G_{n-1}^{\der}$. As we know that $\Gamma_{n-1}\backslash {}^0\mathscr X_{n-1}$ is of finite invariant volume and that there is a split maximal torus in ${}^0G_{n-1}$ over $\RR$, the argument of the proof of \citep[Lemma 1.6]{januszewski2009} shows that $\Gamma_{n-1}$ is eventually Zariski dense in $({}^0G_{n-1})^0$ (this statement is false if $k$ has a complex place, as the classical Theorem of Dirichlet on the structure of the integral units reveals). Furthermore it is easily seen that $G_{n-1}=({}^0G_{n-1})^0 S$. We deduce that there exists a $\Gamma_{n-1}$-linear isomorphism
$$
M_{\check{\boldsymbol{\nu}}}\to M_{\boldsymbol{\mu}^*}
$$
if and only if for some $\nu\in\ZZ$ we have the identity of weights
$$
\check{\boldsymbol{\nu}}+\nu=\boldsymbol{\mu}^*.
$$
For any $\nu$ we have a natural isomorphism of $G_{n-1}$-modules
$$
M_{{\boldsymbol{\nu}}}\otimes M_{\check{\boldsymbol{\nu}}+\nu}\to M_{(\nu)},
$$
where the latter denotes the scheme in free modules of rank $1$, on which $G_{n-1}$ acts via its projection to the torus $G_{n-1}/({}^0G_{n-1})^0$, and the latter acts by the character $s\mapsto s^\nu$. In particular we have natural isomorphisms
$$
H^0\left({\Gamma_{n-1}};
j^*(M_{\boldsymbol{\mu}})\otimes
M_{\boldsymbol{\nu}}
\right)
\cong
H^0\left(({}^0G_{n-1})^0;
j^*(M_{\boldsymbol{\mu}})\otimes
M_{\boldsymbol{\nu}}
\right)
$$
and
\begin{equation}
\tau:H^0\left({\Gamma_{n-1}};
j^*(M_{\boldsymbol{\mu}})\otimes
M_{\boldsymbol{\nu}}
\right)
\to
\bigoplus_{\nu\in{\rm Emb}(\check{\boldsymbol{\nu}},{\boldsymbol{\mu}})}
 M_{(\nu)}
\label{eq:compiso}
\end{equation}
of $G_{n-1}$-modules and likewise a natural equivariant projection
$$
\tau_{\nu}:
H^0\left({\Gamma_{n-1}};
j^*(M_{\boldsymbol{\mu}})\otimes
M_{\boldsymbol{\nu}}
\right)
\to
 M_{(\nu)}.
$$

Write $\underline{M}_{\boldsymbol{\nu}}(A)$ for the sheaf associated to the module $M_{\boldsymbol{\nu}}(A)$ on $\Gamma_{n-1}\backslash\mathscr X_{n-1}$. As $\Gamma_{n-1}\backslash\mathscr X_{n-1}$ is orientable, we can fixing an orientation, which means that we fix an isomorphism $H_{\rm c}^{d_{n-1}}(\Gamma_{n-1}\backslash\mathscr X_{n-1},\underline{A})\cong A$. Assume that $A$ is a noetherian and integrally closed domain. Then any free $A$-module of finite rank is reflexive. Assuming $M_{\boldsymbol{\mu}}(A)$ and $M_{\boldsymbol{\nu}}(A)$ to be free, which is automatic whenever $A$ is a principal ideal domain, we have a natural perfect pairing
$$
\lambda:\left(M_{\boldsymbol{\mu}}(A)\otimes_A M_{\boldsymbol{\nu}}(A)\right)\otimes_A
\left(M_{\boldsymbol{\mu}}(A)\otimes_A M_{\boldsymbol{\nu}}(A)\right)^\vee\to A.
$$
This pairing identifies the right argument as the $A$-dual of the left argument, including the canonical $G_{n-1}(A)$-actions.

Plugging things together we get for any $p\in\ZZ$ a natural pairing
$$
H_{\rm c}^{p}(\Gamma_{n-1}\backslash\mathscr X_{n-1};
\underline{M}_{\boldsymbol{\mu}}(A)\otimes_A
\underline{M}_{\boldsymbol{\nu}}(A))\otimes_A
H^{d_{n-1}-p}(\Gamma_{n-1}\backslash\mathscr X_{n-1}^{\ad};
(\underline{M}_{\boldsymbol{\mu}}(A)\otimes_A
\underline{M}_{\boldsymbol{\nu}}(A))^\vee)\
$$
$$
\to H_{\rm c}^{d_{n-1}}(\Gamma_{n-1}\backslash\mathscr X_{n-1};\underline{A})\to A,
$$
by composing $H_{\rm c}^{d_{n-1}}(\lambda)$ with the $\cup$-product. By Poincar\'e, this is a perfect pairing. In particular we might identify
$$
H_{\rm c}^{d_{n-1}}(\Gamma_{n-1}\backslash\mathscr X_{n-1};
s_\gamma^*(
\underline{M}_{\boldsymbol{\mu}}(A))\otimes_A
\underline{M}_{\boldsymbol{\nu}}(A))
$$
with the $A$-dual of
$$
H^{0}(\Gamma_{n-1}\backslash\mathscr X_{n-1}^{\ad};
(s_\gamma^*(
\underline{M}_{\boldsymbol{\mu}}(A))\otimes_A
\underline{M}_{\boldsymbol{\nu}}(A))^\vee),
$$
which turns out to be\footnote{reflexivity...}
$$
H^{0}(\Gamma_{n-1}\backslash\mathscr X_{n-1}^{\ad};
s_\gamma^*(
\underline{M}_{\boldsymbol{\mu}}(A))\otimes_A
\underline{M}_{\boldsymbol{\nu}}(A)),
$$
taking the $G_{n-1}(A)$-action into account. Hence we get a canonical natural isomorphism
\begin{equation}
\int\limits_{\Gamma_{n-1}\backslash\mathscr X_{n-1}}\!\!\!\!\!\!\!\!\!\!:\quad
H_{\rm c}^{d_{n-1}}(\Gamma_{n-1}\backslash\mathscr X_{n-1};
s_\gamma^*(
\underline{M}_{\boldsymbol{\mu}}(A))\otimes_A
\underline{M}_{\boldsymbol{\nu}}(A))
\to\quad\quad\quad
\label{eq:integration}
\end{equation}
$$
H^{0}(\Gamma_{n-1}\backslash\mathscr X_{n-1};
s_\gamma^*(
\underline{M}_{\boldsymbol{\mu}}(A))\otimes_A
\underline{M}_{\boldsymbol{\nu}}(A)).
$$
Consequently we get by Poincar\'e duality a natural map
$$
{\mathscr B}_\gamma^{\Gamma_n,\Gamma_{n-1}}:\;\;\;
H_{\rm c}^q(\Gamma_{n}\backslash\mathscr X_n^{\ad};
\underline{M}_{\boldsymbol{\mu}}(A))
\otimes_A
H_{\rm c}^{d_{n-1}-q}(\Gamma_{n-1}\backslash\mathscr X_{n-1}^{\ad};\underline{M}_{\boldsymbol{\nu}}(A))
\to
$$
$$
H^0(\Gamma_{n-1}\backslash\mathscr X_{n-1};
s_{\gamma}^*(\underline{M}_{\boldsymbol{\mu}})\otimes
\underline{M}_{\boldsymbol{\nu}}
(A)),
$$
which for $A=\CC$ is given by
$$
\alpha\otimes\beta\;\mapsto\;
\int_{\Gamma_{n-1}\backslash\mathscr X_{n-1}}
s_{\gamma}^*(\alpha)\wedge\ad^*(\beta),
$$
where $d_{n-1}=\dim\mathscr X_{n-1}$.

Note that the map
$$
{}^0\ad:\Gamma_{n-1}\backslash{}^0\mathscr X_{n-1}\to\Gamma_{n-1}\backslash\mathscr X_{n-1}^{\ad}
$$
induced by $\ad$ is proper \citep[Lemma 1.5]{januszewski2009}, and the domain is of finite invariant volume. Furthermore integration along the fibre of the natural decomposition
$$
G_{n-1}(\RR)={}^0G_{n-1}(\RR)\rtimes S(\RR)^0,
$$
induces a natural map
$$
i_*\circ s_\gamma^*:H_{\rm fd}^q(\Gamma_{n}\backslash\mathscr X_{n}^{\ad})\to
H_{\rm mg}^{q-1}(\Gamma_{n-1}\backslash\mathscr X_{n-1}).
$$
Consequently we may identify the above topological symbol with the relative modular symbols constructed in \citep{kazhdanmazurschmidt2000,januszewski2009}. Our definition of the modular symbol is essentially a generalization of that of \citep{schmidt2001}.

We might compose with $\tau_\nu$ to identify the image of $\mathscr B_\gamma^{\Gamma_n,\Gamma_{n-1}}$ as the space $\CC^{{\rm Emb}(\check{\boldsymbol{\nu}},{\boldsymbol{\mu}})}$, because observe that
$$
H^0(\Gamma_{n-1}\backslash\mathscr X_{n-1};
(j^*(\underline{M}_{\boldsymbol{\mu}})\otimes
\underline{M}_{\boldsymbol{\nu}})(\CC))=
H^0(\Gamma_{n-1};
(j^*(M_{\boldsymbol{\mu}})\otimes M_{\boldsymbol{\nu}})(\CC)),
$$
by means of the definition of the sheaves, and conjugation by $\gamma$ induces an automorphism of $M_{\boldsymbol{\mu}}(\CC)$, such that we have an isomorphism
$$
s_\gamma^*(\underline{M}_{\boldsymbol{\mu}}(\CC))\cong
s^*(\underline{M}_{\boldsymbol{\mu}}(\CC)),
$$
and the latter may be canonically identified with a subsheaf of $j^*(\underline{M}_{\boldsymbol{\mu}}(\CC))$. Note however that this is no longer true when we take integral structures into account. Nonetheless the components have the following interpretation.

Pick some compactly supported (or some fast decreasing) differential forms
$$
\alpha^0
=\omega\otimes\phi\in
\left(\bigwedge^q(\lieg_n/(\liek_n+\lies))^*\otimes
\mathscr C^\infty(\Gamma_n\backslash G_n(\RR)/S(\RR)^0; M_{\boldsymbol{\mu}}(\CC))
\right)^{K_n^0}.
$$
We already sketched the construction of how to interpret this as a classical differential form. In particular we set for any $g\in G_n^{\der}(\RR)$
$$
\alpha:=\omega\otimes (g\mapsto \rho_{\boldsymbol{\mu}}(g)(\phi(g))),
$$
then via the identification $\mathscr X_n^{\ad}=\mathscr X_n^{\der}$
$$
\alpha\in\Omega^q(\mathscr X_n^{\ad}; M_{\boldsymbol{\mu}}(\CC))^{\Gamma_{n}}.
$$
We might pushforward this form along the canonical projection to the arithmetic quotient and get a form
$$
\alpha\in\Omega^q(\Gamma_n\backslash\mathscr X_n^{\ad}; \underline{M}_{\boldsymbol{\mu}}(\CC)).
$$
We apply the very same procedure to an element
$$
\beta^0=\omega'\otimes\varphi\in
\left(\bigwedge^{d_n-q}(\lieg_{n-1}/(\liek_{n-1}+\lies))^*\otimes
\mathscr C^\infty(\Gamma_{n-1}\backslash G_{n-1}^{\ad}(\RR); M_{\boldsymbol{\nu}}(\CC))
\right)^{K_{n-1}^0}.
$$
Likewise we get a form
$$
\beta\in\Omega^{d_n-q}(\Gamma_{n-1}\backslash\mathscr X_{n-1}^{\ad}; \underline{M}_{\boldsymbol{\nu}}(\CC)).
$$

The pullback $\overline{\varphi}:\Gamma_{n-1}\backslash\mathscr X_{n-1}\to M_{\boldsymbol{\nu}}(\CC)$ of $\varphi$ along $\Gamma_{n-1}\backslash\mathscr X_{n-1}\to \Gamma_{n-1}\backslash\mathscr X_{n-1}^{\ad}$ is given explicitly by
$$
gz\mapsto \rho_{\boldsymbol{\nu}}(z)^{-1}\varphi(g),
$$
where $z\in \zentrum(G_{n-1})(\RR)^0$. We apply the same notation to the pullback $\overline{\phi}:\Gamma_n\backslash\mathscr X_n\to M_{\boldsymbol{\mu}}(\CC)$ along the analogous map. Then by the isomorphism \eqref{eq:compiso} we see that the above pairing reads
$$
\alpha\otimes\beta\mapsto
\left(\tau\circ\int_{\Gamma_{n-1}\backslash\mathscr X_{n-1}}
\tau_\nu(\rho_{\check{\boldsymbol{\nu}}+\nu}(g)
\otimes
\rho_{\boldsymbol{\nu}}(g))
\left(\overline{\phi}(s_\gamma(g))\otimes\overline{\varphi}(g)\right)
dg
\right)_{\nu\in{\rm Emb}(\check{\boldsymbol{\nu}},{\boldsymbol{\mu}})}
$$
where the right hand side is an element of $\CC^{{\rm Emb}(\boldsymbol{\nu},\check{\boldsymbol{\mu}})}$. The components of the right hand side read
$$
\int_{\Gamma_{n-1}\backslash{}^0\mathscr X_{n-1}}
\int_{S(\RR)^0}
\tau_\nu(\rho_{\check{\boldsymbol{\nu}}+\nu}(gs)\otimes
\rho_{\boldsymbol{\nu}}(gs))(\overline{\phi}(s_\gamma(gs))\otimes\overline{\varphi}(gs))
dsdg
=
$$
$$
\int_{\Gamma_{n-1}\backslash\mathscr X_{n-1}}
\overline{\phi}(s_\gamma(g))\overline{\varphi}(g)s(g)^\nu
dg,
$$
where $s(g)$ denotes the projection of $g$ to the split torus $G_{n-1}/({}^0G_{n-1})^0\cong\Gm$ and the components of the first two formulas are identified with their canonical projections to the $({}^0G_{n-1})^0$-invariants.

\subsection{Construction of cohomology classes}\label{sex:classes}

If we make in particular the following choice for $\phi$ and $\varphi$, we eventually recover the period integrals of the previous section. Assume that we have non-vanishing cohomologies
$$
H^\bullet(\tilde{\lieg}_n,K_n; \pi_\infty\otimes M_{\boldsymbol{\mu}}(\CC))\neq 0,
$$
$$
H^\bullet(\tilde{\lieg}_{n-1},K_{n-1}; \sigma_\infty\otimes M_{\boldsymbol{\nu}}(\CC))\neq 0,
$$
where $\tilde{\lieg}_n$ denotes the complexified Lie algebra of $G_n^{\ad}(\RR)$ or equivalently of $(G_n^\pm)^0$, and $K_n$ is the product of a maximal compact subgroup and the center of $G_n(\RR)$. We tacitly assumed the infinity components to be smooth. Note that this condition of cohomologicality is only slightly stricter than allowing quadratic twists. This assumption simplifies our formulation.

Now for the unique even integer $m\in\{n,n-1\}$ we are eventually interested in the eigen spaces
$$
H^{rb_m}(\tilde{\lieg}_m,K_m^0; \tau_\infty\otimes M_{\boldsymbol{\eta}}(\CC))_\varepsilon\cong\CC,
$$
where $\boldsymbol{\eta}$ is $\boldsymbol{\mu}$ if $m=n$ or $\boldsymbol{\nu}$ otherwise, likewise $\tau_\infty$ is $\pi_\infty$ or $\sigma_\infty$, and $\varepsilon$ is a character of $\pi_0(G_m(\RR))$. These eigen spaces are known to be one-dimensional, where the bottom degree is given by $b_m=\frac{n^2-n+2\left\lfloor\frac{n}{2}\right\rfloor}{2}$, $r=[k:\QQ]$.

Write $m'\in\{n,n-1\}\setminus\{m\}$ for the unique odd rank of interest, $\boldsymbol{\eta}'$ for the corresponding weight and $\tau_\infty'$ for the representation at infinity respectively. In this case
$$
H^{rb_{m'}}(\tilde{\lieg}_{m'},K_m'; \tau_\infty'\otimes M_{\boldsymbol{\eta}'}(\CC))\cong\CC
$$
is one-dimensional and eigen spaces for non-trivial characters vanish here.

So we can pick for each $\varepsilon$ generators of these eigen spaces in
$$
\eta_{\infty,\varepsilon}\in\left(\bigwedge^{rb_m}\tilde{\liep}_m^*\otimes\mathscr{W}_0(\tau_\infty,\psi_\infty)\otimes M_{\boldsymbol{\eta}}(\CC)\right)^{K_m^0},
$$
as the latter space is known to give an explicit description of the $(\tilde{\lieg}_m,K_m^0)$-cohomology of $\tau_\infty$. Similarly we have a generator
$$
\eta_{\infty}'\in\left(\bigwedge^{rb_{m'}}\tilde{\liep}_{m'}^*\otimes\mathscr{W}_0(\tau_\infty',\psi_\infty')\otimes M_{\boldsymbol{\eta}'}(\CC)\right)^{K_{m'}}.
$$
The relative cohomologies with respect to $K_{m'}$ and $K_{m'}^0$ are here the same due to our assumption on cohomologicality. We apologize for the imprecise notation concerning the additive characters $\psi_\infty$ and $\psi_\infty'$. We assume them to be compatible with our previous choices, i.e.\ they are dual to each other in an appropriate sense.

Now we can add local Whittaker functions of finite places to build up global Whittaker functions. To be more precise, choose Maurer-Cartan bases of $\tilde{\liep}_{m}$ and $\tilde{\liep}_{m'}$ and write with respect to those
$$
\eta_{\infty,\varepsilon}=\sum_{\#I=rb_m}\omega_I\otimes w_{\infty,\varepsilon,I},
$$
where $w_{\infty,\varepsilon,I}\in\mathscr{W}_0(\tau_\infty,\psi_\infty)\otimes M_{\boldsymbol{\eta}'}(\CC)$. Then we tensor the latter Whittaker functions with our finite component $w_f\in\mathscr{W}(\tau_f,\psi_f)$ and we get a well defined element
$$
\sum_{\#I=rb_m}\omega_I\otimes w_{\varepsilon,I}=
\sum_{\#I=rb_m}\omega_I\otimes w_{\infty,\varepsilon,I}\otimes w_{f},
$$
which by Fourier transform yields a cuspidal class
$$
[\sum_{\#I=rb_m}\omega_I\otimes \phi_{\varepsilon,I}]\in H^{rb_m}(\tilde{\lieg}_m,K_m^0;
L_0^2(G_m(\QQ)\backslash G_m(\Adeles_\QQ)/K_{m,f}; M_{\boldsymbol{\eta}}(\CC)))_{\varepsilon},
$$
for some compact open $K_{m,f}\subseteq G_m(\Adeles_\QQ^{(\infty)})$, which in turn yields, thanks to Borel \citep{borel1974,borel1981}, a cohomology class with compact support
$$
[\eta_\varepsilon]\in H^{rb_m}_{\rm c}(G_m(\QQ)\backslash G_m(\Adeles_\QQ)/K_{m,f}K_m^0; \underline{M}_{\boldsymbol{\eta}}(\CC))_{\varepsilon}
$$
(note that this transition implicitly requires the formalism we discussed before, and this is eventually compatible with the component action, i.e.\ the resulting class lies again in the $\varepsilon$-eigen space as desired).

The very same procedure works for $\eta_\infty'$ and yields a class\footnote{in this case the component action is trivial.}
$$
[\eta']\in H^{rb_{m'}}_{\rm c}(G_{m'}(\QQ)\backslash G_{m'}(\Adeles_\QQ)/K_{m',f}K_{m'}^0; \underline{M}_{\boldsymbol{\eta}'}(\CC)).
$$

Now, depending on whether $m=n$ or not we choose $\alpha_\varepsilon=[\eta_\varepsilon]$ or $\alpha_\varepsilon=[\eta']$, and $\beta_\varepsilon=[\eta']$ or $\beta_\varepsilon=[\eta_\varepsilon]$ respectively. Then we might plug in the universal class
$$
\lambda_{\pi,\sigma}:=\sum_{\varepsilon}\alpha_\varepsilon\otimes\beta_\varepsilon
$$
into our modular symbol and we eventually get the desired periods (for some suitable choice of $\gamma$), modulo constants depending on the parity of the critical value, which might vanish. By classical results of Shimura and Manin in the case $n=2$ and Mazur, Schmidt and Kasten-Schmidt in the case $n=3$, we know that, depending on the signs of the components of $\chi_\infty$ and the parity of $\nu$, there is precisely one $\varepsilon$ such that $\alpha_\varepsilon\otimes\beta_\varepsilon$ contributes to the evaluation of the modular symbol at $\lambda_{\pi,\sigma}$ and the corresponding period $\Omega(\sum_{I\cap I'=\emptyset}w_{\infty,\varepsilon,I}\otimes v_{\infty,\varepsilon,I'},\chi_\infty)$ does not vanish at $\frac{1}{2}+\nu$.

Now rationality and even integrality follows if we choose our local Whittaker vectors in such a way that the resulting class becomes integral, i.e.\ a non-zero element\footnote{eventually the integral structure in cohomology with compact supports induces integral structures on our eigen spaces in Lie algebra cohomology, which in turn as a unique generator up to integral units. We assume that we choose such a generator for any $\varepsilon$.} of
$$
H^{rb_{m}}_{\rm c}(G_m(\QQ)\backslash G_m(\Adeles_\QQ)/K_{m,f}K_m^0; \underline{M}_{\boldsymbol{\eta}}(\OO_E))_{\varepsilon},
\;\;\;\text{resp.}\;\;\;
$$
$$
H^{rb_{m'}}_{\rm c}(G_{m'}\QQ)\backslash G_{m'}(\Adeles_\QQ)/K_{m',f}K_{m'}^0; \underline{M}_{\boldsymbol{\eta}'}(\OO_E)),
$$
where $\OO_E$ is the ring of integers of a field of rationality $E$ for $(\pi,\sigma)$. It is well known that this choice is possible by virtue of a generalized Eichler-Shimura isomorphism. Eichler-Shimura guarantees that a suitable choice at the finite places yields an integral class modulo scalars. So the essential point in our construction is multiplicity one in bottom degree, i.e.\ the corresponding eigen spaces are one-dimensional, and an appropriate scaling gives rise to the periods (under the widely believed conjecture that there is a good vector at infinity supporting cohomology; in some sense the work of Kasten-Schmidt can be interpreted as a partial result in this direction in the case $n=3$; furthermore results of Harder and Mahnkopf also point in this direction, modulo scalars in $E$).

After all, we have an, up to integral units and possible torsion, well-defined integral element
$$
\lambda_{\pi,\sigma}\;\in\;
H^{rb_{n}}_{\rm c}(G_n(\QQ)\backslash G_n(\Adeles_\QQ)/K_{n,f}K_n^0; \underline{M}_{\boldsymbol{\mu}}(\OO_E))
$$
$$
\otimes_{\OO_E}H^{rb_{n-1}}_{\rm c}(G_{n-1}(\QQ)\backslash G_{n-1}(\Adeles_\QQ)/K_{n-1,f}K_{n-1}^0; \underline{M}_{\boldsymbol{\nu}}(\OO_E)).
$$
Extending our modular symbol to the various connected components we get for each component a vector
$$
{\mathscr B}_\gamma^{\Gamma_n,\Gamma_{n-1}}(\lambda_{\pi,\sigma})\;\in\;
H^0(
({}^0G_{n-1})^0;
j_{\gamma}^*M_{\boldsymbol{\mu}}(\OO_E)\otimes_{\OO_E}
M_{\boldsymbol{\nu}}(\OO_E)),
$$
where
$$
j_{\gamma}:G_{n-1}\to G_{n},\;\;\;g\mapsto \gamma^{-1}j(g),
$$
encoding the period integrals from the global Birch Lemma.

\subsection{Integral structures on local systems}

Consider the $G_{n}(\QQ)$-module $M_{\boldsymbol{\mu}}(E)$ for any field extension $E/\QQ(\boldsymbol{\mu})$. It gives rise to a sheaf $\underline{M}_{\boldsymbol{\mu}}(E)$ on the locally symmetric space
$$
\mathscr X_n(K):=G_{n}(\QQ)\backslash G_{n}(\Adeles_\QQ)/K_\infty^0K,
$$
where $K$ is any compact open subgroup of $G_{n}(\Adeles_\QQ^{(\infty)})$ and $K_\infty\subseteq G_{n-1}(\RR)$ is the product of the standard maximal compact subgroup and the center of $G_{n-1}(\RR)$. We henceforth assume that for all $g\in G_n(\Adeles_\QQ^{(\infty)})$
$$
\Gamma_g:=G_{n}(\QQ)\cap gKg^{-1}
$$
is torsion free. Then $\mathscr X_n(K)$ is a manifold and with $\Gamma:=\Gamma_1$ we may identify
$$
\Gamma\backslash\mathscr X_n\subseteq\mathscr X_n(K)
$$
naturally as the connected component of the the origin. Furthermore this corresponds by strong approximation to the fiber above the identity of the surjective determinant map
$$
{\det}_K:\mathscr X_n(K)\to k^\times\backslash\Adeles_k^\times/(k\otimes_\QQ\RR)^0\det(K).
$$
Note that the base
$$
C(K):=k^\times\backslash\Adeles_k^\times/(k\otimes_\QQ\RR)^0\det(K).
$$
is finite. In particular we write
$$
\mathscr X_n(K)[c]:={\det}_K^{-1}(c)
$$
for the fiber of a class
$$
c\in C(K).
$$
We often assume that $c\in \Adeles_k^\times$ is also a representative, that we may and do assume to be a finite id\`ele, i.e.\ trivial at infinity.

In order to examine integral structures we realize the sheaf $\underline{M}_{\boldsymbol{\mu}}(E)$ explicitly as
$$
\Gamma(U;\underline{M}_{\boldsymbol{\mu}}(E))=
\{f:G_n(\QQ)U\to M_{\boldsymbol{\mu}}(E)\mid f\;\text{locally constant and}
$$
$$
\forall\gamma_{n}\in ({}^0G_n)^0(\QQ),u\in G_n(\QQ)U:f(\gamma_{n}) u)=\rho_{\boldsymbol{\mu}}(\gamma_{n})f(u)
\}.
$$
Now let $\OO\subseteq E$ be a subring admitting $E$ as quotient field and let $M_{\boldsymbol{\mu}}(\OO)$ be a $K$-stable $\OO$-lattice in $M_{\boldsymbol{\mu}}(E)$, i.e.\
$$
M_{\boldsymbol{\mu}}(\OO)\otimes_\OO E\cong M_{\boldsymbol{\mu}}(E)
$$
naturally and
$$
\rho_{\boldsymbol{\mu}}(K)(M_{\boldsymbol{\mu}}(\OO)\otimes_\ZZ\hat{\ZZ})\subseteq M_{\boldsymbol{\mu}}(\OO)\otimes_\ZZ\hat{\ZZ}.
$$
We have a sheaf $\underline{M}_{\boldsymbol{\mu}}^0(\OO)$ on $\mathscr X_n(K)[1]$, explicitly given by
$$
\Gamma(U;\underline{M}_{\boldsymbol{\mu}}^0(\OO))=
\{f:\Gamma U\to M_{\boldsymbol{\mu}}(\OO)\mid f\;\text{locally constant and}
$$
$$
\forall\gamma\in \Gamma,u\in \Gamma U:f(\gamma) u)=\rho_{\boldsymbol{\mu}}(\gamma)f(u)
\}.
$$
Any $g\in G_n(\Adeles_\QQ^{(\infty)})$ induces a diffeomorphism
$$
t_g^0:\mathscr X_n(gKg^{-1})[c]\to \mathscr X_n(K)[c\det(g)],
$$
by translation
$$
G_n(\QQ)x gKg^{-1}\mapsto G_n(\QQ)xg K,
$$
which only depends on the class $gK$. Now let
$$
gM_{\boldsymbol{\mu}}(\OO):=
M_{\boldsymbol{\mu}}(E)\cap \rho_{\boldsymbol{\mu}}(g)(M_{\boldsymbol{\mu}}(\OO)\otimes_\ZZ\hat{\ZZ}),
$$
where the intersection takes place in
$$
M_{\boldsymbol{\mu}}(E)\otimes_\QQ\Adeles_\QQ^{(\infty)}.
$$
Then this sheaf is stable under $gKg^{-1}$ and for the arithmetic group
$$
\Gamma_g:=G_n(\QQ)\cap gKg^{-1}
$$
we have the associated sheaf
$$
\underline{gM}_{\boldsymbol{\mu}}^0(\OO).
$$
on $\mathscr X_n(gKg^{-1})[1]$. Then we get the sheaf
$$
\underline{M}_{\boldsymbol{\mu}}^0(\OO)[\det(g)]:=t_{g,*}^0\underline{gM}_{\boldsymbol{\mu}}^0(\OO)
$$
on $\mathscr X_n(K)[\det(g)]$. By strong approximation it does only depend on the class of $\det(g)$ in $C(K)$.

Plugging the above sheaves together we get a subsheaf $\underline{M}_{\boldsymbol{\mu}}(\OO)$ of $\underline{M}_{\boldsymbol{\mu}}(E)$, given by
$$
\Gamma(U;\underline{M}_{\boldsymbol{\mu}}(\OO))=
\{f\in\Gamma(U;\underline{M}_{\boldsymbol{\mu}}(E)\mid
\forall c\in C(K):
$$
$$
f|_{U\cap\mathscr X_n(K)[c]}\in
\Gamma(U\cap\mathscr X_n(K)[c];
\underline{M}_{\boldsymbol{\mu}}^0(\OO)[c])
\}.
$$
This sheaf may be naturally identified with the topological sum of the sheaves $t_{g_c,*}^0\underline{g_cM}_{\boldsymbol{\mu}}^0(\OO)$ for a system of representatives $g_c$ with $\det(g_c)=c$ running through $C(K)$.

In the very same spirit we may for any $g\in G_n(\Adeles_k^{(\infty)})$ identify the pullback of $\underline{M}_{\boldsymbol{\mu}}(\OO)$ along the translation by $g$ map
$$
t_g:\mathscr X_n(gKg^{-1})\to\mathscr X_n(K)
$$
with an analoguously defined sheaf $\underline{gM}_{\boldsymbol{\mu}}(\OO)$, which itself is naturally a subsheaf of $\underline{M}_{\boldsymbol{\mu}}(E)$ on $\mathscr X_n(gKg^{-1})$. To keep track of the various realizations, we write
$$
T_g:t_g^*\underline{M}_{\boldsymbol{\mu}}(\OO)\to
\underline{gM}_{\boldsymbol{\mu}}(\OO)
$$
for the natural isomorphism. Sometimes we consider $t_{g,*}T_g$ also as an isomorphism
$$
t_{g,*}T_g:\underline{M}_{\boldsymbol{\mu}}(\OO)\to
t_{g,*}\underline{gM}_{\boldsymbol{\mu}}(\OO)
$$
of sheaves on $\mathscr X_n(K)$. By construction we have a canonical inclusion
$$
i_g:\underline{gM}_{\boldsymbol{\mu}}(\OO)\to\underline{M}_{\boldsymbol{\mu}}(E).
$$
If $g$ turns out to stabilize $M_{\boldsymbol{\mu}}(\OO)$ we consider $i_g$ also as an embedding
$$
i_g:\underline{gM}_{\boldsymbol{\mu}}(\OO)\to\underline{M}_{\boldsymbol{\mu}}(\OO),
$$
induced by the set-theoretic interpretation of both sides as subsheaves of $\underline{M}_{\boldsymbol{\mu}}(E)$. For notational efficiency we introduce the abbreviations
$$
iT_g:=i_g\circ T_g,
$$
$$
Tt_g^*:=T_g\circ t_g^*,
$$
and
$$
iTt_g^*:=i_g\circ T_g\circ t_g^*.
$$
We note that
\begin{equation}
(T_x\circ t_x^*)\circ (T_y\circ t_y^*)=T_{xy}\circ t_{xy}^*
\label{eq:Ttfunctor}
\end{equation}
and similarly for $i_g$. This formalism carries over to the various kinds of cohomology and integration commutes with $T_g$ and $t_g^*$ and $i_g$ in the appropriate way.

\subsection{Hecke operators on cohomology}

The action of the Hecke algebra of finite level on cohomology may be most conceptually defined via Hecke correspondences. Another approach comes from the cohomology of groups, by interpreting the direct image with compact supports that occurs in the construction via correspondences as a trace map for groups. We need an explicit description of the action, that we put also in the foreground of our definitions. This also leads directly to the instances of the Eichler-Shimura isomorphism that we need.

Let $K\subseteq G_n(\Adeles_\QQ^{(\infty)})$ be compact open and pick an element $g\in G_n(\Adeles_\QQ^{(\infty)})$. Then the double coset represented by $g$ decomposes into finitely many right cosets
$$
KgK=\bigsqcup_i g_iK.
$$
On the sheaf $\underline{M}_{\boldsymbol\mu}(E)$ on $\mathscr X_n(K)$ we define an action sending a section
$$
s\in\Gamma(U;\underline{M}_{\boldsymbol{\mu}}(E))
$$
to
$$
s|_{[KgK]}:=\sum_{i}iTt_{g_i}^*(s)
$$
again on the space $\mathscr X_n(K)$. This action extends to cohomology, i.e.\ considering the right derived functors of the global sections functor we get an endomorphism
$$
\cdot|_{[KgK]}\in\End_E(H^q_*(\mathscr X_n(K);\underline{M}_{\boldsymbol\mu}(E)))
$$
for any $*\in\{-,\rm c,!\}$.

Now assume that $K$ is of Iwahori level $I_n^{(m)}$ at $\mathfrak{p}$. Then $\mathcal H_{I}\otimes_\QQ E$ embeds into $\mathcal H_{E}(K,G_n(\Adeles_\QQ^{(\infty)}))$. Hence as $\mathcal H_{E}(K,G_n(\Adeles_\QQ^{(\infty)}))$ acts on the cohomology from the left, $\mathcal H_{I}$ acts on cohomology from the left and this action is compatible with the action on automorphic forms, i.e.\ we have a generalized Eichler-Shimura map.

\subsection{Cohomological definition of the distribution}

Let for any $\mathfrak{p}$-power ideal $\mathfrak{f}$ in $k$
$$
C(\mathfrak{f}):=
k^\times\backslash\Adeles_k^\times/(1+\mathfrak{f})\prod_{\mathfrak{q}\nmid \mathfrak{p}}U_\mathfrak{q},
$$
$$
C(\mathfrak{p}^\infty):=\varprojlim C(\mathfrak{f})=
k^\times\backslash\Adeles_k^\times/\prod_{\mathfrak{q}\nmid\mathfrak{p}}U_\mathfrak{q},
$$
where $\mathfrak{f}$ ranges over the $\mathfrak{p}$-power ideals. On these groups we have a natural action of the id\`eles $\Adeles_k^\times$.

Fix some compact open subgroups $K$ and $K'$ in $\GL_n(\Adeles_k^{(\infty)})$ resp.\ $\GL_{n-1}(\Adeles_k^{(\infty)})$, which are factorizable and coincide with $I_n^{(m)}$ resp.\ $I_{n-1}^{(m)}$ at $\mathfrak{p}$ for a fixed $m\geq 1$, and satisfy $j(K')\subseteq K$ and $\det(K)=\det(K')$. Denote $\Gamma$ resp.\ $\Gamma'$ the corresponding arithmetic groups, and assume that for all $g\in G_n(\Adeles_\QQ^{(\infty)})$ the arithmetic groups
$$
G_n(\QQ)\cap gKg^{-1}
$$
are torsion free. We assume the same statement for $K'$. We fix $f=\varpi^{\nu_\mathfrak{p}(\mathfrak{f})}$, a generator of $\mathfrak{f}$. Set
$$
K(\mathfrak{f}):=j^{-1}(h^{(f)}K(h^{(f)})^{-1})\cap K'.
$$
Then for $\nu_\mathfrak{p}(\mathfrak{f})\geq m$ it is known that the $\mathfrak{p}$-component of $\det(K(\mathfrak{f}))$ equals $1+\mathfrak{f}$, cf.\ \citep[Proposition 3.4]{schmidt2001}. Therefore we have a finite covering map
\begin{equation}
C(K(\mathfrak{f}))\to C(\mathfrak{f})
\label{eq:Cfcover}
\end{equation}
with the notation of the previous section.

Let $K_\infty$ be the product of the standard maximal compact subgroup in $G_{n-1}(\RR)$ and the center of the latter. Consider the locally symmetric space
$$
\mathscr X(\mathfrak{f}):=\GL_{n-1}(k)\backslash\GL_{n-1}(\Adeles_k)/K(\mathfrak{f})K_\infty^0.
$$
We have a surjection
$$
{\det}_{K(\mathfrak{f})}:\mathscr X(\mathfrak{f})
\to C(K(\mathfrak{f})),
$$
induced by the determinant map. The fibers of ${\det}_{K(\mathfrak{f})}$ are, by strong approximation, given by the translates of
$$
\mathscr X(\mathfrak{f})[1]={\det}_{K(\mathfrak{f})}^{-1}(1)\cong \Gamma(\mathfrak{f})\backslash\mathscr X_{n-1}
$$
where $\Gamma(\mathfrak{f})\subseteq\GL_{n-1}(k)$ is the arithmetic subgroup corresponding to $K(\mathfrak{f})$.

Consider the proper map
$$
s_{h^{(f)}}:\mathscr X(\mathfrak{f})\to\mathscr X_{n}^{\rm ad}(K),
$$
which results from the composition of the embedding $j$ with the translation $t_{h^{(f)}}$ and the adjoint map.

We define topologically for any $h\in G_n(\Adeles_\QQ^{(\infty)})$ and any $x\in \Adeles_k^{(\infty)\times}$ the topological modular symbol
$$
{\mathscr B}_{h,x}^{K,K'}:
H_{\rm c}^{rb_n}(\mathscr X_n^{\ad}(K); \underline{M}_{\boldsymbol{\mu}}(\OO))\otimes_\OO
H_{\rm c}^{rb_{n-1}}(\mathscr X_{n-1}^{\ad}(K'); \underline{M}_{\boldsymbol{\nu}}(\OO))\to
$$
$$
H^0(\mathscr X_{n-1}(j^{-1}(hKh^{-1})\cap K')[x];
s_{h}^*\underline{M}_{\boldsymbol{\mu}}(\OO)\otimes_\OO
\underline{M}_{\boldsymbol{\nu}}(\OO)),
$$
via
$$
\lambda=\sum_{\varepsilon}\alpha_\varepsilon\otimes\beta_\varepsilon\mapsto
\int\limits_{\mathscr X_{n-1}(j^{-1}(hKh^{-1})\cap K')[x]}
\sum_{\varepsilon}
s_{h}^*\alpha_\varepsilon \cup \ad^*\beta_\varepsilon.
$$
Assuming that $\lambda$ is an eigen class for the Hecke operator
$$
U_\mathfrak{p}:=
V_\mathfrak{p}\otimes V_\mathfrak{p}'
$$
with eigen value $\kappa_\lambda\in E^\times$ (i.e.\ $\lambda$ is of finite slope) we set
$$
\kappa_\lambda(\mathfrak{f}):=
\kappa_\lambda^{-\nu_\mathfrak{p}(\mathfrak{f})},
$$
and for any $\mathfrak{f}$ with $\nu_\mathfrak{p}(\mathfrak{f})\geq m$, we considering $x+\mathfrak{f}$ as an element of $C(K(\mathfrak{f}))$, and define
$$
\mu_\lambda(x+\mathfrak{f}):=
\kappa_\lambda(\mathfrak{f})\cdot
(iTt_{d_{(x)}f t_{(f)}}^*)
(s_1^*(iT_{h^{(f)}})\otimes 1)
{\mathscr B}_{h^{(f)},xf^{\frac{n(n-1)}{2}}}^{
K,
K'}
(\lambda),
$$
as an element of
$$
H^0(
({}^0G_{n-1})^0;
j^*M_{\boldsymbol{\mu}}(E)\otimes_E
M_{\boldsymbol{\nu}}(E)),
$$
the latter space being independent of the levels, and hence independent of $\mathfrak{f}$ and $x$. To see this observe first that the translation related operators commute with integration. In particular we see that
$$
\mu_\lambda(x+\mathfrak{f})=
\kappa_\lambda(\mathfrak{f})\cdot
\!\!\!\!\!\!\!\!\!\!\!\!\!\!\!\!\!\!\!\!\!\!\!\!
\int\limits_{\mathscr X_{n-1}(d_{(x)}t_{(f)}K(\mathfrak{f})t_{(f)}^{-1}d_{(x)}^{-1})[1]}
\!\!\!\!\!\!\!\!\!\!\!\!\!\!\!\!\!\!
(iTt_{d_{(x)}f t_{(f)}}^*)
(s_1^*(iTt_{h^{(f)}})\otimes 1)
\sum_{\varepsilon}s_1^*(\alpha_\varepsilon)\cup\ad^*(\beta_\varepsilon),
$$
and the integrand lives in
$$
H_{\rm c}^{d_{n-1}}(\mathscr X_{n-1}(d_{(x)}t_{(f)}K(\mathfrak{f})t_{(f)}^{-1}d_{(x)}^{-1});
s_1^*
\underline{M}_{\boldsymbol{\mu}}(E)\otimes_E \underline{M}_{\boldsymbol{\nu}}(E)).
$$
Now our claim follows as the arithmetic subgroup corresponding to $d_{(x)}t_{(f)}K(\mathfrak{f})t_{(f)}^{-1}d_{(x)}^{-1}$ on the component with determinant $1$ is Zariski dense in $({}^0G_{n-1})^0$.

\subsection{The distribution relation}

In this section we study the effect of the Hecke operator $U_\mathfrak{p}$ on our modular symbol and prove the distribution relation if $\lambda$ is an eigen vector for this operator with non-zero eigen value $\kappa_\lambda\in E^\times$.

To begin with, we need to refine of a known result of Schmidt about the relation of the matrices $h^{(f)}$ and $h^{(f\varpi)}$ (cf.\ \citep[Lemma 3.2]{schmidt2001} or \citep[Lemma B.0.1]{januszewski2009}).

For any $x\in\Adeles_k^{(\infty)}$ and any $u\in U_{n}(\OO_{\mathfrak{p}}),$ $w\in U_{n-1}(\OO_{\mathfrak{p}})$ set
$$
h(u,w)\;:=\;
t_{(f)}j(w)t_{(f)}^{-1}
\cdot h^{(1)}\cdot
t_{(f)}u^{-1}t_{(f)}^{-1}
\;\in\; G_n(\OO_\mathfrak{p}/\mathfrak{fp}).
$$
This defines a group action on a subset of $G_n(\OO_\mathfrak{p}/\mathfrak{fp})$. Writing $u=(u_{ij})$ and $w=(w_{ij})$ then $h(u,w)$ only depends on the terms
$$
u_{12},u_{23},\dots,u_{n-1 n}
$$
and
$$
w_{12},w_{23},\dots,w_{n-2 n-1}.
$$
Note that we have for the same reason
$$
t_{(f)}u^{-1}t_{(f)}^{-1}\;\equiv\; t_{(f)}(-u_{ij})_{ij}t_{(f)}^{-1}\pmod{\mathfrak{fp}}.
$$
We conclude that $h(u,w)$ equals the matrix
$$
\begin{pmatrix}
          0 &  \hdots&  \hdots&     0 &fw_{12}&           1 &1+fw_{12}-fu_{n-1n}\\
      \vdots&        &  \adots&fw_{23}&     1 &-fu_{n-2 n-1}&1+fw_{23}          \\
      \vdots&  \adots&  \adots& \adots& \adots&           0 &1+fw_{34}\\
          0 &  \adots&  \adots& \adots& \adots&       \vdots&   \vdots\\
fw_{n-2 n-1}&      1 &-fu_{23}& \adots&       &       \vdots&1+fw_{n-2 n-1}\\
          1 &-fu_{12}& \adots &       &       &       \vdots&         1    \\
          0 &      0 & \hdots & \hdots& \hdots&            0 &        1    \\
\end{pmatrix}
$$
Define the lower triangular unipotent matrices
$$
u^-:=
\begin{pmatrix}
         1 &  &&&\\
fu_{n-2n-1}&1 &&&\\
           &\ddots & \ddots &  &\\
           & &fu_{32} & 1 & \\
           & &   & fu_{21} & 1\\
\end{pmatrix}\in t_{(f)}^{-1}U_{n-1}^-(\OO_\mathfrak{p})t_{(f)},
$$
and
$$
w^-:=
\begin{pmatrix}
         1 &  &&&\\
-fw_{n-2n-1}&1 &&&\\
           &\ddots & \ddots &  &\\
           & &-fw_{32} & 1 & \\
           & &   & -fw_{21} & 1\\
\end{pmatrix}\in t_{(f)}^{-1}U_{n-1}^-(\OO_\mathfrak{p})t_{(f)}.
$$
Then the first $n-1$ columns of the matrix
$$
u^-\cdot h(u,w)\cdot w^-
$$
equal those of $h^{(1)}$ and the last column is the tranpose of
$$
(1+fw_{12}-fu_{n-1n}, 1+fw_{23}-fu_{n-2n-1}, \dots,1+fw_{n-2n-1}-fu_{23},1-fu_{12},1).
$$
Set
$$
d(u,w):=\diag(1-fu_{12},1+fw_{n-2n-1}-fu_{23},\dots,1+fw_{12}-fu_{n-1n})
$$
and
$$
d'(u,w):=\diag(1-fw_{12}+fu_{n-1n},\dots,1-fw_{n-2n-1}+fu_{23},1+fu_{12}).
$$
Note that
$$
\det(d(u,w))\cdot\det(d'(u,w))=1.
$$
\begin{lemma}\label{lem:distributionmatrixrelation}
For any $u,w$ as above there exist matrices $k_{u,w}'\in I_{n-1}^{(m)}$ and $k_{u,w}\in I_{n}^{(m)}$ with the property that
$$
t_{(\varpi)}^{-1} j(w)\cdot h^{(f)}\cdot u^{-1} t_{(\varpi)}=j(k')\cdot h^{(f\varpi)}\cdot k^{-1}
$$
and that furthermore sending $u,w$ to
$$
\det(k_{u,w})=\det(k_{u,w}')\pmod{\mathfrak{fp}}
$$
defines an epimorphism of groups
$$
U_n(\OO_\mathfrak{p})/t_{(\varpi)}U_n(\OO_\mathfrak{p})t_{(\varpi)}^{-1}\times
U_{n-1}(\OO_\mathfrak{p})/t_{(\varpi)}U_{n-1}(\OO_\mathfrak{p})t_{(\varpi)}^{-1}
\to
1+\mathfrak{f}/1+\mathfrak{fp}.
$$
\end{lemma}

\begin{proof}
By the above discussion there is a matrix $n\in I_n^{(m)}$ with
$$
n\equiv{\bf1}_n\pmod{\mathfrak{fp}}
$$
and
$$
j(d'(u,w))\cdot u^-\cdot h(u,w)\cdot w^-\cdot d(u,w)\cdot n=h^{(1)}.
$$
Now observe that
$$
t_{(\varpi)}^{-1}\cdot j(w)\cdot h^{(f)}\cdot u^{-1}\cdot t_{(\varpi)}=
t_{(f\varpi)}^{-1}\cdot h(u,w)\cdot t_{(f\varpi)}.
$$
Therefore the choice
$$
k_{u,w}':=t_{(f\varpi)}^{-1}\cdot j(d'(u,w))\cdot u^-\cdot t_{(f\varpi)}\in I_{n-1}^{(m)}
$$
$$
k_{u,w}:=\left(t_{(f\varpi)}^{-1}\cdot j(w^-\cdot d(u,w))\cdot n\cdot t_{(f\varpi)}\right)^{-1}\in I_{n}^{(m)}
$$
proves the claim.
\end{proof}

As an application we have
\begin{lemma}\label{lem:distributionintegralrelation}
For any
$$
\alpha\in H^{rb_{n}}_{\rm c}(\mathscr X_{n}^{\rm ad}(K); \underline{M}_{\boldsymbol{\mu}}(E))
$$
and any
$$
\beta\in H^{rb_{n-1}}_{\rm c}(\mathscr X_{n-1}^{\rm ad}(K'); \underline{M}_{\boldsymbol{\nu}}(E))
$$
and any $u\in U_n(\OO_\mathfrak{p})$ and $w\in U_{n-1}(\OO_\mathfrak{p})$ we have
$$
(s_1^*(iT_{h^{(f)}})\otimes 1)
\mathscr B_{h^{(f)},xf^{\frac{n(n-1)}{2}}}^{
ut_{(\varpi)}K(ut_{(\varpi)})^{-1}, wt_{(\varpi)}K'(wt_{(\varpi)})^{-1}}
\!\!\!\!\!
(iTt_{ut_{(\varpi)}}^*\alpha\otimes
iTt_{w\varpi t_{(\varpi)}}^*\beta)
=
$$
$$
iTt_{d_{(\det(k_{u,w^{-1}}))}\varpi t_{(\varpi)}}^*\circ
(s_1^*iT_{h^{(f\varpi)}}\otimes 1)
\mathscr B_{h^{(f\varpi)},x\det(k_{u,w^{-1}})(f\varpi)^{\frac{n(n-1)}{2}}}^{
K,K'}
(\alpha\otimes\beta)
$$
\end{lemma}

\begin{proof}
Under the above hypothesis we have
$$
\mathscr B_{h^{(f)},xf^{\frac{n(n-1)}{2}}}^{
ut_{(\varpi)}K(ut_{(\varpi)})^{-1}, wt_{(\varpi)}K'(wt_{(\varpi)})^{-1}}
(iTt_{ut_{(\varpi)}}^*\alpha\otimes
iTt_{w\varpi t_{(\varpi)}}^*\beta)=
$$
$$
\int\limits_{\mathscr X_{n-1}(j^{-1}(
h^{(f)}ut_{(\varpi)}
K
(h^{(f)}ut_{(\varpi)})^{-1}
)
\cap
wt_{(\varpi)}K'(wt_{(\varpi)})^{-1})
[xf^{\frac{n(n-1)}{2}}]}
\!\!\!\!\!\!\!\!\!\!\!\!\!\!\!\!\!\!\!\!\!\!\!\!\!\!\!\!\!\!\!\!\!\!\!\!\!\!\!\!\!\!\!\!\!\!\!\!\!\!\!\!
s_{h^{(f)}}^*(iTt_{ut_{(\varpi)}}^*\alpha)
\cup
\ad^*(iTt_{w\varpi t_{(\varpi)}}^*\beta).
$$
Now due to \eqref{eq:Ttfunctor} we have an identity of pullbacks
$$
iTt_{(w \varpi t_{(\varpi)})^{-1}}^*=
iTt_{(\varpi t_{(\varpi)})^{-1}}^*,
$$
because $w\in{}^0G_{n-1}$. Hence the composition of this pullback with $s_{1}^*iT_{h^{(f)}}\otimes 1$ maps the above section to
$$
(
s_1^*iT_{t_{(\varpi)}^{-1}j(w)^{-1}h^{(f)}ut_{(\varpi)}}\otimes 1)
\!\!\!\!\!\!\!\!\!\!\!\!\!\!\!\!\!\!\!\!\!\!\!\!\!\!\!\!\!\!\!\!\!\!\!\!\!\!\!\!\!\!\!\!\!\!\!\!\!\!\!\!\!\!\!\!
\int\limits_{
\mathscr X_{n-1}(
j^{-1}(
t_{(\varpi)}^{-1}j(w)^{-1}h^{(f)}ut_{(\varpi)}
K
(t_{(\varpi)}^{-1}j(w)^{-1}h^{(f)}ut_{(\varpi)})^{-1}
)
\cap
K')[x(f\varpi)^{\frac{n(n-1)}{2}}]}
\!\!\!\!\!\!\!\!\!\!\!\!\!\!\!\!\!\!\!\!\!\!\!\!\!\!\!\!\!\!\!\!\!\!\!\!\!\!\!\!\!\!\!\!\!\!\!\!\!\!\!\!\!\!\!\!\!\!\!\!
s_{t_{(\varpi)}^{-1}j(w)^{-1}h^{(f)}u t_{(\varpi)}}^*\alpha
\cup
\ad^*\beta.
$$
Now by Lemma \ref{lem:distributionmatrixrelation} we have the elements $k'=k_{u^{-1},w}'\in K'$ and $k=k_{u^{-1},w}\in K$ with $\det (k')=\det(k)\in 1+\mathfrak{f}$, such that
$$
j(k')^{-1}\cdot
h^{(f\varpi)}\cdot
k
=
t_{(\varpi)}^{-1}
j(w)^{-1}\cdot
h^{(f)}\cdot
u t_{(\varpi)}.
$$
Consequently the pullback $iTt_{k'}^*$ maps the above section to
$$
(s_1^*
iT_{h^{(f\varpi)}}\otimes 1)
\!\!\!\!\!\!\!\!\!\!\!\!\!\!\!\!\!\!\!\!\!\!\!\!\!\!\!\!\!\!\!\!\!\!\!\!\!\!\!\!\!\!
\int\limits_{
\mathscr X_{n-1}(
j^{-1}(
h^{(f\varpi)}
K
(h^{(f\varpi)})^{-1}
)
\cap
K'
)[x
\det(k)^{-1}
(f\varpi)^{\frac{n(n-1)}{2}}]}
\!\!\!\!\!\!\!\!\!\!\!\!\!\!\!\!\!\!\!\!\!\!\!\!\!\!\!\!\!\!\!\!\!\!\!\!\!\!\!\!\!\!
s_{h^{(f\varpi)}}^*(iTt_{k}^*\alpha)
\cup
\ad^*(iTt_{k'}^*\beta).
$$
We have
$$
iTt_{k}^*(\alpha)=\alpha
$$
and similarly for $\beta$. Now we observe that
$$
d_{(\det(k)^{-1})}\cdot k'\in({}^0G_{n-1})^0,
$$
as this element has determinant equal to $1$. Hence
$$
iTt_{k'}^*=iTt_{d_{(\det(k))}}^*,
$$
and the claim follows.
\end{proof}

We are now in a position to prove

\begin{theorem}\label{thm:distribution}
For any $x\in\Adeles_k^{(\infty)}$ and any ideal $\mathfrak{f}=\mathfrak{p}^{\nu_\mathfrak{p}(\mathfrak{f})}$ with $\nu_\mathfrak{p}(\mathfrak{f})\geq m$ we have
$$
\mu_\lambda(x+\mathfrak{f})=\sum_{a\!\!\!\pmod{\mathfrak{p}}}
\mu_\lambda(x+af+\mathfrak{f}\mathfrak{p}).
$$
In particular $\mu_\lambda$ is a distribution on $C(K(\mathfrak{p}^\infty))=\varprojlim\limits_{\mathfrak{f}}C(K(\mathfrak{f}))$ with values in
$$
H^0(({}^0G_{n-1})^0; j^*M_{\boldsymbol{\mu}}(E)\otimes_E M_{\boldsymbol{\nu}}(E)).
$$
\end{theorem}

\begin{proof}
By our hypothesis $\lambda$ is an eigen vector for
$$
U_\mathfrak{p}=
V_\mathfrak{p}\otimes V_\mathfrak{p}',
$$
which by Lemma \ref{lem:hecke1} means that
$$
\kappa_\lambda
\cdot\lambda=
\sum_{\varepsilon,u,w}
(iTt_{ut_{(\varpi)}}^*\alpha_\varepsilon)\otimes
(iTt_{w\varpi t_{(\varpi)}}^*\beta_\varepsilon),
$$
where both sides are interpreted as cohomology classes with respect to the intersections $\tilde{K}$ resp.\ $\tilde{K}'$ of the compact open subgroups $ut_{(\varpi)}K(ut_{(\varpi)})^{-1}$ for all $u$ and $wt_{(\varpi)}K'(wt_{(\varpi)})^{-1}$ for all $w$ respectively. By Lemmata \ref{lem:distributionmatrixrelation} and \ref{lem:distributionintegralrelation} this yields the decomposition
$$
[K(\mathfrak{f}):K(\mathfrak{fp})]\cdot
(s_1^*iT_{h^{(f)}}\otimes 1)
\!\!\!\!\!
\int\limits_{\mathscr X(\mathfrak{f})[xf^{\frac{n(n-1)}{2}}]}
\!\!\!\!\!
(s_{h^{(f)}}^*\alpha\cup \ad^*\beta)=
$$
$$
\frac{(U_n(\OO_{k,\mathfrak{p}}):t_{(\varpi)}U_n(\OO_{k,\mathfrak{p}})t_{(\varpi)}^{-1})\cdot
(U_{n-1}(\OO_{k,\mathfrak{p}}):t_{(\varpi)}U_{n-1}(\OO_{k,\mathfrak{p}})t_{(\varpi)}^{-1})}
{\absNorm(\mathfrak{p})}\cdot
$$
$$
\kappa_\lambda\cdot
\!\!\!\!\!
\sum_{a\!\!\!\pmod{\mathfrak{p}}}
iTt_{\varpi t_{(\varpi)}}^*\circ
(s_1^*iT_{h^{(f\varpi)}}\otimes 1)
\!\!\!\!\!
\int\limits_{\mathscr X(\mathfrak{fp})[(x+af)(f\varpi)^{\frac{n(n-1)}{2}}]}
\!\!\!\!\!
(s_{h^{(f\varpi)}}^*\alpha\cup \ad^*\beta).
$$
By the following known index formulas
\begin{equation}
(I_{n-1}^{(m)}:K(\mathfrak{f}))=
\absNorm(\mathfrak{f})^{\frac{(n+1)n(n-1)+n(n-1)(n-2)}{6}}
\label{eq:gammaindex}
\end{equation}
cf.\ \citep[Lemmata 3.6 and 3.7]{schmidt2001}, and
$$
(U_n(\OO_{k,\mathfrak{p}}):t_{(f)}U_n(\OO_{k,\mathfrak{p}})t_{(f)}^{-1})=\absNorm(\mathfrak{f})^{\frac{(n+1)n(n-1)}{6}},
$$
cf.\ \citep[Proof of Lemma 3.2]{kazhdanmazurschmidt2000}, the claim follows.
\end{proof}

\subsection{Boundedness in the ordinary case}

Let $E/\QQ(\boldsymbol{\mu},\boldsymbol{\nu})$ be a number field. Fix an embedding $i_{\mathfrak{p}}:E\to\overline{k}_\mathfrak{p}$. We write $\absnorm{\cdot}_\mathfrak{p}$ for the norm on $E$ induced by $i_{\mathfrak{p}}$ and let $\OO_{E,(\mathfrak{p})}\subseteq E$ denote its valuation ring. Then if $\lambda$ is an eigen vector for $\kappa_\lambda\in E$ we say that $\lambda$ is {\em ordinary at $\mathfrak{p}$} if
\begin{equation}
\absnorm{\kappa_\lambda}_\mathfrak{p}=
%\absNorm(\mathfrak{p})^{-\nu_{\min}\cdot\frac{n(n-1)}{2}}.
\absnorm{\varpi}^{\nu_{\min}\cdot\frac{n(n-1)}{2}}.
\label{eq:lambdaordinarity}
\end{equation}
We say that $\lambda$ is {\em of finite slope} if $\kappa_\lambda\neq 0$ and then the integer
$$
\nu_\mathfrak{p}\left(\kappa_\lambda\cdot\varpi^{-\nu_{\min}\cdot\frac{n(n-1)}{2}}\right)\in\ZZ
$$
is called the {\em slope} of $\lambda$ at $\mathfrak{p}$. So $\lambda$ is ordinary if and only if it is of slope $0$.

We know that the associated sheaves for $M_{\boldsymbol{\mu}}(\OO_{E,(\mathfrak{p})})$ and $M_{\boldsymbol{\nu}}(\OO_{E,(\mathfrak{p})})$ are stable under $G_{n}(\Adeles_\QQ^{(\infty p)})$ resp.\ $G_{n-1}(\Adeles_\QQ^{(\infty p)})$. We assume that
\begin{equation}
j^*M_{\boldsymbol{\mu}}(\OO_{E,(\mathfrak{p})})\otimes_{\OO_{E,(\mathfrak{p})}}
M_{\boldsymbol{\nu}}(\OO_{E,(\mathfrak{p})})\otimes_{\OO_{E,(\mathfrak{p})}}
M_{(-\nu_{\min})}(\OO_{E,(\mathfrak{p})})
\label{eq:stability}
\end{equation}
is stable under $\mathfrak{p}$-integral matrices, i.e.\ under the action of
$$
G_{n-1,\mathfrak{p}}:=
\{g\in\GL_{n-1}(k)\mid g\in\OO_{k,{(\mathfrak{p})}}^{n-1\times n-1}\}.
$$
Integral models with this property clearly exist, as we know from section \ref{sec:cohint} that the module \eqref{eq:stability} decomposes into
$$
\bigoplus_{\nu\in{\rm Emb}(\check{\boldsymbol{\nu}},\boldsymbol{\mu})}
M_{\boldsymbol{\check{\nu}}}
\otimes
M_{\boldsymbol{\nu}}[\nu-\nu_{\min}]
$$
and $\nu\geq \nu_{\min}$.

\begin{theorem}\label{thm:boundedness}
If $\lambda$ is ordinary at $\mathfrak{p}$ and $\mathfrak{p}$-integral, then $\mu_\lambda$ takes values in
$$
H^0(
({}^0G_{n-1})^{0};
j^*M_{\boldsymbol{\mu}}(\OO_{E,(\mathfrak{p})})\otimes_{\OO_{E,(\mathfrak{p})}}
M_{\boldsymbol{\nu}}(\OO_{E,(\mathfrak{p})})).
$$
\end{theorem}

\begin{proof}
First we observe that
$$
d_{(x)}\cdot t_{(f)}\cdot h^{(f)}=d_{(x)}\cdot h^{(1)}\cdot t_{(f)}.
$$
We might choose $x\in\GL_1(\Adeles_k^{(\infty)})$ in such a way that $x_v\in\OO_{k,v}^\times$ for all $v\mid p$.
Now $d_{(x)}ft_{(f)}$ acts on $M_{(-\nu_{\min})}(\OO_{E,(\mathfrak{p})})$ via multiplication by
$$
f_0:=\left(xf^{\frac{n(n-1)}{2}}\right)^{-\nu_{\min}}.
$$
By our choice of $x$ and the ordinarity condition \eqref{eq:lambdaordinarity}
$$
\absnorm{\kappa_\lambda(\mathfrak{f})f_0^{-1}}_{\mathfrak{p}}=1.
$$
We deduce that we have an identity
$$
\kappa_\lambda(\mathfrak{f})\cdot
\underline{(d_{(x)}ft_{(f)})^{-1}M}_{(-\nu_{\min})}(\OO_{E,(\mathfrak{p})})=
\underline{M}_{(-\nu_{\min})}(\OO_{E,(\mathfrak{p})})
$$
of subsheaves of $\underline{M}_{(-\nu_{\min})}(E)$. In particular we have the integral global section
\begin{equation}
\gamma:=\kappa_\lambda(\mathfrak{f})\cdot Tt_{(d_{(x)}ft_{(f)})^{-1}}^*1\in M_{(-\nu_{\min})}(\OO_{E,(\mathfrak{p})}).
\label{eq:kappaintegrality}
\end{equation}
Consider the natural isomorphism
$$
r:H^0(\mathscr X_{n-1}(d_{(x)}t_{(f)}K(\mathfrak{f})t_{(f)}^{-1}d_{(x)}^{-1})[1];
s_1^*
\underline{M}_{\boldsymbol{\mu}}(E)
\otimes_{E}
\underline{M}_{\boldsymbol{\nu}}(E)
)
[-\nu_{\min}]
\to
$$
$$
H^0(\mathscr X_{n-1}(d_{(x)}t_{(f)}K(\mathfrak{f})t_{(f)}^{-1}d_{(x)}^{-1})[1];
s_1^*
\underline{M}_{\boldsymbol{\mu}}(E)
\otimes_{E}
\underline{M}_{\boldsymbol{\nu}}(E)
[-\nu_{\min}]
),
$$
which respects integral structures in the obvious way and commutes with integration. In particular for any choice of cohomology classes
$$
\alpha\otimes\beta\in H^{rb_{n}}_{\rm c}(\mathscr X_{n}^{\rm ad}(K); \underline{M}_{\boldsymbol{\mu}}(\OO_E))
\otimes_{\OO_E} H^{rb_{n-1}}_{\rm c}(\mathscr X_{n-1}^{\rm ad}(K'); \underline{M}_{\boldsymbol{\nu}}(\OO_E))
$$
the element
$$
\kappa_\lambda(\mathfrak{f})\cdot
\bigg(
Tt_{d_{(x)}ft_{(f)}}^*
(s_1^*T_{h^{(f)}}\otimes 1)
\!\!\!\!\!\!\!\!\!\!\!\!
\int\limits_{\mathscr X(\mathfrak{f})[xf^{\frac{n(n-1)}{2}}]}
\!\!\!\!\!\!\!\!\!\!\!\!
s_{h^{(f)}}^*\alpha\cup\ad^*\beta\bigg)\otimes 1
$$
maps, by \eqref{eq:kappaintegrality}, under $r$ to
\begin{equation}
Tt_{d_{(x)}ft_{(f)}}^*
(s_1^*T_{h^{(f)}}\otimes 1\otimes 1)
\!\!\!\!\!\!\!\!\!\!\!\!\!\!\!\!
\int\limits_{\mathscr X(\mathfrak{f})[xf^{\frac{n(n-1)}{2}}]}
\!\!\!\!\!\!\!\!\!\!\!\!\!\!\!\!
(s_{h^{(f)}}^*\alpha\cup(\ad^*\!\beta\otimes\gamma))
\in
H^0(\mathscr X_{n-1}(d_{(x)}t_{(f)}K(\mathfrak{f})t_{(f)}^{-1}d_{(x)}^{-1})[1];
\label{eq:alphabetagamma}
\end{equation}
$$
$$
$$
s_1^*
\underline{d_{(x)}h^{(1)}t_{(f)}M}_{\boldsymbol{\mu}}(\OO_{E,(\mathfrak{p})})
\otimes
\underline{d_{(x)}ft_{(f)}M}_{\boldsymbol{\nu}}(\OO_{E,(\mathfrak{p})})
\otimes
\underline{d_{(x)}ft_{(f)}M}_{(-\nu_{\min})}(\OO_{E,(\mathfrak{p})})
).
$$
Now as
$$
d_{(x)}f\cdot t_{(f)}\in \overline{G}_{n-1,\mathfrak{p}},
$$
is $\mathfrak{p}$-integral, hence lies in the $\mathfrak{p}$-adic closure of $\mathfrak{p}$-integral matrices, the module \eqref{eq:stability} is table under this element, in particular the section \eqref{eq:alphabetagamma} lies in
$$
H^0(\mathscr X(d_{(x)}t_{(f)}K(\mathfrak{f})t_{(f)}^{-1}d_{(x)}^{-1})[1];
s_1^*
\underline{d_{(x)}h^{(1)}d_{(x)}^{-1}M}_{\boldsymbol{\mu}}(\OO_{E,(\mathfrak{p})})
\otimes
\underline{M}_{\boldsymbol{\nu}}(\OO_{E,(\mathfrak{p})})
\otimes
\underline{M}_{(-\nu_{\min})}(\OO_{E,(\mathfrak{p})})
).
$$
As the isomorphism $r$ respects integral structures, we conclude that
$$
\mu(x+\mathfrak{f})\in
H^0(\mathscr X_{n-1}(d_{(x)}t_{(f)}K(\mathfrak{f})t_{(f)}^{-1}d_{(x)}^{-1})[1];
s_1^*
\underline{d_{(x)}h^{(f)}d_{(x)}^{-1}M}_{\boldsymbol{\mu}}(\OO_{E,(\mathfrak{p})})
\otimes
\underline{M}_{\boldsymbol{\nu}}(\OO_{E,(\mathfrak{p})})).
$$
Finally, due to our choice of $x$ we have 
$$
d_{(x)}\cdot h^{(1)}\cdot d_{(x)}^{-1}\;\in\; K\cdot G_{n}(\Adeles_\QQ^{(\infty p)}),
$$
hence $M_{\boldsymbol{\mu}}(\OO_{E,(\mathfrak{p})})$ is stable under this matrix, concluding the proof.
\end{proof}

We remark that the same proof yields an explicit bound on the order of the resulting distribution in the case of positive finite slope.

\subsection{The interpolation formula}

Now let $\pi$ and $\sigma$ denote irreducible cuspidal automorphic representations of $\GL_n$ and $\GL_{n-1}$ over $k$ respectively, possessing non-zero $I_n^{(m)}$ resp.\ $I_{n-1}^{(m)}$-invariant vectors at $\mathfrak{p}$. Assume that $\pi$ and $\sigma$ are regular algebraic with cohomological coefficients as in section \ref{sec:cohint}. Then their finite parts $\pi^{(\infty)}$ and $\sigma^{(\infty)}$ are defined over a number field $E=\QQ(\pi,\sigma)$ \citep[Th\'eor\`eme 3.13 resp.\ Proposition 3.16]{clozel1990}.

Note that the Hecke polynomial $H_\mathfrak{p}$ eventually lies in $\mathcal H_{I_n^{(m)}}[X]$. Choose Hecke roots
$$
\lambda_{1},\dots,\lambda_{n}\in E,
$$
for $\pi_\mathfrak{p}$ in the sense of \eqref{eq:heckeroots}, i.e.\ $H_\mathfrak{p}(\lambda_\nu)$ annihilate a non-zero vector $w_\mathfrak{p}^0$ in the Whittaker model of $\pi_\mathfrak{p}$, and similarly Hecke roots
$$
\lambda_{1}',\dots,\lambda_{n-1}'\in E,
$$
for $\sigma_\mathfrak{p}$ annihilating a vector $v_\mathfrak{p}^0$. Let
$$
\underline{\lambda}:=
(\lambda_{1},\dots,\lambda_{n-1})\in E^{n-1},
$$
and
$$
\underline{\lambda}':=
(\lambda_{1},\dots,\lambda_{n-1})\in E^{n-1}.
$$
We remark that we include one more eigen value in $\underline{\lambda}'$ as in the case of trivial central character. We set
$$
\lambda'':=(\lambda_1',\dots,\lambda_{n-2}').
$$
With this notation let
$$
\kappa_{\underline{\lambda}}:=
\absNorm(\mathfrak{p})^{-\frac{n(n-1)(n-2)}{6}}\cdot
\prod_{\nu=1}^{n-1}
\lambda_{\mathfrak{p},\nu}^{n-\nu}.
$$
and
$$
\kappa_{\underline{\lambda}'}:=
\absNorm(\mathfrak{p})^
{-\frac{n(n-1)(n-2)}{6}}
\cdot
\prod_{\nu=1}^{n-1}
\lambda_{\mathfrak{p},\nu}'^{n-\nu}.
$$

We call the tuple $(\pi,\sigma,\underline{\lambda},\underline{\lambda}')$ {\em of finite slope at} $\mathfrak{p}$ if the following three conditions hold:
\begin{itemize}
\item[(i)]
$V_{\mathfrak{p},n-1}$ acts on $\sigma_\mathfrak{p}$ via the scalar
\begin{equation}
\eta_{n-1}=\absNorm(\mathfrak{p})^{-\frac{(n-1)(n-2)}{2}}
\cdot
\prod_{\nu=1}^{n-1}\lambda_\nu'.
\label{eq:zentralscalar}
\end{equation}
\item[(ii)]
The vectors $w_\mathfrak{p}^0$ and $v_\mathfrak{p}^0$ may be chosen in such a way that
$$
\Pi_{\underline{\lambda}}^0(w_\mathfrak{p}^0)({\bf1}_n)=
\Pi_{\underline{\lambda}''}^0(v_\mathfrak{p}^0)({\bf1}_{n-1})=
\prod_{\nu=1}^{n-1}\left(1-\absNorm(\mathfrak{p})^{-\nu}\right).
$$
\item[(iii)] The {\em slope}
$$
\nu_\mathfrak{p}\left(
\kappa_{\underline{\lambda}}\cdot\kappa_{\underline{\lambda}'}
\cdot\varpi^{-\nu_{\min}\cdot\frac{n(n-1)}{2}}\right)\in\ZZ\cup\{\infty\},
$$
is finite (i.e.\ $\lambda_n$ might well be zero).
\end{itemize}
If in addition the slope is $0$, we call the datum {\em ordinary}.

Assuming that the Whittaker vectors satisfy condition (ii), we set
$$
\tilde{w}_\mathfrak{p}:=\Pi_{\underline{\lambda}}^0(w_\mathfrak{p}^0),
$$
and
$$
\tilde{v}_\mathfrak{p}:=\Pi_{\underline{\lambda}}^0(w_\mathfrak{p}^0).
$$

Assume that the cohomology class $\lambda_{\pi,\sigma}$ was constructed as in section \ref{sex:classes}, where the local Whittaker vectors at $\mathfrak{p}$ where chosen as $w_\mathfrak{p}^0$ and $v_\mathfrak{p}^0$ as above. We define
$$
\tilde{\lambda}_{\pi,\sigma,\underline{\lambda},\underline{\lambda}'}:=\Pi_{\underline{\lambda}}^0\otimes\Pi_{\underline{\lambda}''}^0(\lambda_{\pi,\sigma}).
$$
By Proposition \ref{prop:heckemodifikation} and the condition on the action of $V_{\mathfrak{p},n-1}$ on $\sigma_\mathfrak{p}$, this is an eigen vector for the Hecke operator $U_\mathfrak{p}$ with eigen value $\kappa_{\underline{\lambda}}\cdot\kappa_{\underline{\lambda}'}$. Then if $\underline{\lambda}$ and $\underline{\lambda}'$ are ordinary, so is $\tilde{\lambda}_{\pi,\sigma,\underline{\lambda},\underline{\lambda}'}$.

We remark that this modification is compatible with modification on the automorphic side, i.e.\ $\tilde{\lambda}_{\pi,\sigma,\underline{\lambda},\underline{\lambda}'}$ corresponds to the collection of automorphic forms $\tilde{\phi}_\iota$ and $\tilde{\varphi}_\iota$, which are constructed by applying $\Pi_{\underline{\lambda}}^0$ resp.\ $\Pi_{\underline{\lambda}''}^0$ to the automorphic forms $\phi_\iota$ resp.\ $\varphi_\iota$ corresponding to $\lambda_{\pi,\sigma}$.

The following theorem strengthens and generalizes \citep[Theorem 4.4]{januszewski2009}.

\begin{theorem}\label{thm:interpolation}
Assume that $(\pi,\sigma,\underline{\lambda},\underline{\lambda}')$ is of finite slope. Then for any character $\chi:k^\times\backslash\Adeles_k^\times\to\CC^\times$ of finite order with non-trivial $\mathfrak{p}$-power conductor $\mathfrak{f}_\chi$, we have the interpolation formula
$$
\tau\circ\int\limits_{C(\mathfrak{p}^\infty)}
\chi d\mu_{\tilde{\lambda}_{\pi,\sigma,\underline{\lambda},\underline{\lambda}'}}\;=\;
$$
$$
\left(
\Omega(w_\infty\otimes v_\infty,\chi_\infty)(\frac{1}{2}+\nu)\cdot
\hat{\kappa}^\nu(\mathfrak{f}_\chi)\cdot
G(\chi)^{\frac{n(n-1)}{2}}\cdot
L^{(\mathfrak{p})}(\frac{1}{2}+\nu,(\pi\times\sigma)\otimes\chi)
\right)_{\nu\in{\rm Emb}(\check{\boldsymbol{\nu}},\boldsymbol{\mu})}.
$$
Here $\hat{\kappa}^\nu(\mathfrak{f}_\chi)$
is given explicitly by
$$
\hat{\kappa}^\nu(\mathfrak{f}_\chi):=
\absNorm(\mathfrak{f}_\chi)^{\frac{n(n-1)(n-2)}{6}+(\nu-\nu_{\min})\frac{n(n-1)}{2}}
\cdot
(
\kappa_{\underline{\lambda}}\cdot
\kappa_{\underline{\lambda}'}
)^{-\nu_{\mathfrak{p}}(\mathfrak{f}_\chi)}
.
$$
\end{theorem}

For $\pi_\mathfrak{p}$ resp.\ $\sigma_\mathfrak{p}$ spherical, with pairwise distinct Hecke roots, it is well known that the corresponding data are all of finite slope (cf.\ \citep[Proposition 4.12]{kazhdanmazurschmidt2000}).

Up to the computation of the Euler factors at the finite places $v\nmid\mathfrak{p}$ where $\det(K)\neq\OO_v^\times$, Theorem \ref{thm:interpolation} immediately generalizes to arbitrary finite order characters of $C(K(\mathfrak{p}^\infty))$ with $\mathfrak{p}$ in its conductor, cf.\ \eqref{eq:Cfcover}.

\begin{proof}
We may choose $\mathfrak{f}$ in such a way that$\nu_\mathfrak{p}(\mathfrak{f})\geq m$ and $\mathfrak{f}_\chi\mid\mathfrak{f}$. Then
$$
\int\limits_{C(\mathfrak{p}^\infty)}\chi d\mu_{\tilde{\lambda}_{\pi,\sigma}}=
\sum_{x\in C(\mathfrak{f})}\chi(x)\mu_{\tilde{\lambda}_{\pi,\sigma,\underline{\lambda},\underline{\lambda}'}}(x+\mathfrak{f}),
$$
and the $\nu$-th component, after composing with $\tau$, is (up to computable indices)
$$
\sum_{\iota,x\in C(\mathfrak{f})}\chi(x)
\int\limits_{C_{\mathfrak{f}}}
\tilde{\phi}_\iota
\left(
j(gd_{(x)}ft_f)
\cdot
h^{(f)}
\right)
\cdot
\tilde{\varphi}_\iota(gd_{(x)}ft_f)
\absnorm{\det(gd_{(x)}ft_f)}^{\nu}
dg,
$$
by the description of the period integrals given in section \ref{sec:cohint}. Writing $C_{\mathfrak{f}}[f_\chi^{\frac{n(n-1)}{2}}]$ for the corresponding fiber with determinant $f_\chi^{\frac{n(n-1)}{2}}$ the right invariance of the Haar measure yields that the period integral in question equals
$$
\absNorm(\mathfrak{ff}_\chi^{-1})^{-\nu\frac{n(n-1)}{2}}
\!\!\!\!\!\!\!\!\!\!\!\!
\int\limits_{C_{\mathfrak{f}}[f_\chi^{\frac{n(n-1)}{2}}]}
\!\!\!\!\!\!\!\!\!\!\!
\tilde{\phi}_\iota
\left(
j(gd_{(x)}) t_{(ff_\chi^{-1})}
h^{(f)}
\right)
\tilde{\varphi}_\iota(gd_{(x)}ff_\chi^{-1}t_{(ff_\chi^{-1})})
\absnorm{\det(gd_{(x)})}^{\nu}
dg.
$$
Now $\tilde{v}_\mathfrak{p}$ is an eigen vector for the operator $V'_{\mathfrak{p}}$ with eigen value $\kappa_{\underline{\lambda}'}$, furthermore $\psi_\mathfrak{p}$ is unramified, and therefore, writing $\delta:=\nu_\mathfrak{p}(\mathfrak{f}\mathfrak{f}_\chi^{-1})$,
$$
v_\mathfrak{p}(ff_\chi^{-1}t_{(ff_\chi^{-1})})=
\absNorm(\mathfrak{ff}_\chi^{-1})^{-\frac{n(n-1)}{2}}\cdot V_{\mathfrak{p}}'^\delta v({\bf1}_{n-1})=
$$
$$
\absNorm(\mathfrak{ff}_\chi^{-1})^{-\frac{(n-1)(n-2)}{2}}\cdot 
\sum_u
v_\mathfrak{p}(uff_\chi^{-1}t_{(ff_\chi^{-1})})=
\absNorm(\mathfrak{ff}_\chi^{-1})^{\frac{(n-1)(n-2)}{2}}\cdot \kappa_{\underline{\lambda}'}^\delta \cdot v_\mathfrak{p}({\bf1}_{n-1}),
$$
where $u\in U_{n-1}(\OO_\mathfrak{p})$ runs through a system of representatives in the sense of Lemma \ref{lem:hecke1}.  An analogous argument applies to $\tilde{w}_\mathfrak{p}$ and Corollary \ref{kor:globalbirch}, together with the known index \eqref{eq:gammaindex}, then concludes the proof.
\end{proof}

\section{The functional equation}

In order to establish the functional equation, we need to introduce compatible notions of contragredience in various settings. This formalism is more involved than in the classically known low-dimensional case $n=2$.

\subsection{Contragredient Hecke modules}

We use the notation of section \ref{sec:hecke} and start with considering the full level. Consider the twisted main involution
$$
\iota:g\mapsto w_ng^{-t}w_n
$$
of $\GL_n$, where the supscript $-t$ denotes matrix inversion composed with transpose. This is an outer automorphism of order $2$. Let $\mathcal M$ be a vector space over a field $E$ with a left action of the Hecke algebra $\mathcal H_{I}$ of Iwahori level $I_n^{(m)}$.

We consider the full Hecke algebra $\mathcal H_{G}$ as embedded into $\mathcal H_{I}$. Now as $\iota$ stabilizes the corresponding Hecke pairs, it induces outer automorphisms of $\mathcal H_{G}$ and $\mathcal H_{I}$, commuting with the embedding. It also stabilizes the pair $(K_B,B)$ and commutes with the embeddings into $\mathcal H_{B}$.

We have the $\iota$-twisted $\mathcal H_{I}$-module $\mathcal M^\vee$. It comes with a canonical map
$$
\mathcal M\to \mathcal M^\vee,\;\;\;m\mapsto m^\vee,
$$
which is twisted $\mathcal H_{I}$-invariant.

Let $m\in \mathcal M$. $T_n$ acts on $m$ by a scalar $c\in E^\times$, then $T_n$ acts on
$$
m^\vee\in\mathcal M^\vee
$$
via the inverse scalar $c^{-1}$. More generally, the action of the Hecke operator $T_\nu$ on $\mathcal M^\vee$ is given by
$$
T_\nu m^\vee=T_n(T_{n-\nu}m)^\vee.
$$
Using this relation we get
\begin{proposition}\label{prop:heckereciproc}
If for some $\lambda\in E^\times$
$$
H_\mathfrak{p}(\lambda)m=0,
$$
then
$$
\lambda^\vee:=
\absNorm(\mathfrak{p})^{n-1}
\lambda^{-1}
$$
is a Hecke root for $m^\vee$, i.e.
$$
H_\mathfrak{p}(\lambda^\vee)m^\vee=0.
$$
\end{proposition}

\begin{proof}
If we have for some $\lambda\in E^\times$
$$
H_\mathfrak{p}(\lambda)m=0,
$$
then
$$
H_\mathfrak{p}(
\absNorm(\mathfrak{p})^{n-1}
\lambda^{-1}
)m^\vee=
\sum_{\nu=0}^n
(-1)^\nu
\absNorm(\mathfrak{p})^{\frac{(\nu-1)\nu}{2}}
(\absNorm(\mathfrak{p})^{n-1}\lambda^{-1})^{n-\nu}
T_\nu m^\vee=
$$
$$
\absNorm(\mathfrak{p})^{\frac{n(n-1)}{2}}
\lambda^{-n}\sum_{\nu=0}^n
\sum_{\nu=0}^n
(-1)^\nu
\absNorm(\mathfrak{p})^{\frac{(\nu-1)\nu}{2}+\frac{n(n-1)}{2}-\nu(n-1)}
\lambda^{\nu}
T_n(T_{n-\nu} m)^\vee.
$$
Now an easy calculation shows that
\begin{equation}
\frac{(\nu-1)\nu}{2}+\frac{n(n-1)}{2}-\nu(n-1)=
\frac{(n-\nu-1)(n-\nu)}{2}.
\label{eq:recisums}
\end{equation}
Therefore
$$
\absNorm(\mathfrak{p})^{\frac{n(n-1)}{2}}
(-\lambda)^{-n}
T_n\left(
\sum_{\nu=0}^n
(-1)^{\nu}
\absNorm(\mathfrak{p})^{\frac{(\nu-1)\nu}{2}}
\lambda^{n-\nu}
T_{\nu} m\right)^\vee=0,
$$
proving the claim.
\end{proof}

Assume that $\mathcal M$ is an $\mathcal H_{I}\times\mathcal H_{I'}$-module, where
$$
\mathcal H_{I'}=\mathcal H_\QQ(I_{n-1}^{(m)},I_{n-1}^{(m)}{T'}^+I_{n-1}^{(m)})
$$
is the Iwahori Hecke algebra for $\GL_{n-1}$. Define the inclusions
$$
i:\mathcal H_{I}\to\mathcal H_{I}\otimes\mathcal H_{I'},\;T\mapsto T\otimes 1
$$
and
$$
i':\mathcal H_{I'}\to\mathcal H_{I}\otimes\mathcal H_{I'},\;T\mapsto 1\otimes T.
$$
Similarly we define the contragredient module $\mathcal M^\vee$ by twisting with $\iota\otimes\iota$ and define the map $\cdot^\vee:\mathcal M\to\mathcal M^\vee$ as before.

Let $m\in M$ have Hecke roots $\lambda_1,\dots,\lambda_{n-1}\in E$ for $\mathcal H_I$ i.e.\
$$
i(H_\mathfrak{p}(\lambda_\nu))\cdot m=0
$$
for $1\leq\nu\leq n-1$. Similarly let $m$ have Hecke roots $\lambda_1',\dots,\lambda_{n-1}'\in E$ for $\mathcal H_{I'}$. We set
$$
\underline{\lambda}:=(\lambda_1,\dots,\lambda_{n-1})
$$
and
$$
\underline{\lambda}':=(\lambda_1',\dots,\lambda_{n-1}')
$$
as before.

Then we say that $(m,\underline{\lambda},\underline{\lambda}',\boldsymbol{\mu},\boldsymbol{\nu})$ is {\em of finite slope} if $i(T_n)$ acts on $m$ via a non-zero scalar, $i'(T_{n-1})$ acts on $m$ via the scalar \eqref{eq:zentralscalar} and if furthermore
$$
\nu_\mathfrak{p}\left(
\kappa_{\underline{\lambda}}\cdot\kappa_{\underline{\lambda}'}
\right)\in\ZZ\cup\{\infty\},
$$
is finite. If the slope is $0$, we call the datum {\em ordinary}. We remark that in the finite slope case we find a unique $\lambda_n\in E^\times$ such that $i(T_n)$ acts via the scalar
$$
\eta_n=\absNorm(\mathfrak{p})^{-\frac{n(n-1)}{2}}\prod_{\nu=1}^n\lambda_\nu\in E^\times,
$$
For a datum of finite slope we set
$$
\underline{\lambda}^\vee:=(\lambda_{n}^\vee,\dots,\lambda_2^\vee)
$$
and
$$
\underline{\lambda}'^\vee:=(\lambda_{n-1}'^\vee,\dots,\lambda_1'^\vee)
$$
in the notation of Proposition \ref{prop:heckereciproc}.

Note that we have in $\mathcal H_{I}$ the identity
$$
T_n=\absNorm(\mathfrak{p})^{-\frac{(n-1)n}{2}}\prod_{i=1}^n U_i=V_{\mathfrak{p},n}
$$
by Gritsenko's factorization of $H_\mathfrak{p}$. For the operators $V_{\mathfrak{p},\nu}$ and $V_\mathfrak{p}$ we have by Lemma \ref{lem:hecke1}
\begin{equation}
V_{\mathfrak{p},\nu} m^\vee =
V_{\mathfrak{p},n}
(V_{\mathfrak{p},n-\nu} m)^\vee.
\label{eq:twistedVpn}
\end{equation}
for $0\leq \nu\leq n$ and
\begin{equation}
V_\mathfrak{p} m^\vee =
V_{\mathfrak{p},n}^{n-1}
(V_\mathfrak{p} m)^\vee.
\label{eq:twistedVp}
\end{equation}
Hence again if $V_\mathfrak{p}$ acts as via multiplication by a unit $\eta\in E^\times$ on $m$, then $V_\mathfrak{p}$ acts on $m^\vee$ via a unit $\eta^\vee\in E^\times$ if and only if $V_{\mathfrak{p},n}=T_n$ acts via a unit $c\in E^\times$ on $m$.

In the finite slope case the hypotheses of Proposition \ref{prop:heckemodifikation} are fulfilled,
and we have the dual relation
$$
H_\mathfrak{p}(\lambda_\nu^\vee)m^\vee=0
$$
for $1\leq \nu\leq n-1$. Under this condition we have the two modified vectors
$
m_{\underline{\lambda}},
$
resp.
$
(m^\vee)_{\underline{\lambda}^\vee},
$
both eigen vectors of the operator $V_\mathfrak{p}$ with the respective eigen values
$$
\eta:=
\absNorm(\mathfrak{p})^{-\frac{n(n-1)(n-2)}{6}}\cdot
\prod_{\nu=1}^{n-1}\lambda_\nu^{n-\nu},
$$
and
$$
\eta^\vee:=
\absNorm(\mathfrak{p})^{-\frac{n(n-1)(n-2)}{6}}\cdot
\prod_{\nu=1}^{n-1}(\lambda_{n+1-\nu}^\vee)^{n-\nu}.
$$
A direct calculation shows that $V_\mathfrak{p}$ acts on
$$
(m_{\underline{\lambda}})^\vee\in\mathcal M^\vee
$$
via $\eta^\vee$. Similarly $V_{\mathfrak{p},n}$ acts on $m_{\underline{\lambda}^\vee}^\vee$ via
$$
\eta_n^\vee=
\absNorm(\mathfrak{p})^{-\frac{n(n-1)}{2}}\prod_{\nu=1}^n\lambda_\nu^\vee=
\absNorm(\mathfrak{p})^{\frac{n(n-1)}{2}}\prod_{\nu=1}^n\lambda_\nu^{-1}=\eta_n^{-1}.
$$
We have
\begin{proposition}\label{prop:contragredient}
Let $\mathcal M$ be an $\mathcal H_{I}\times\mathcal H_{I'}$-module and let $m\in\mathcal M$ be a vector with Hecke roots $\lambda_1,\dots,\lambda_{n-1}$ for $i(\mathcal H_{I})$ and with Hecke roots $\lambda_1',\dots,\lambda,_{n-1}'$ for $i'(\mathcal H_{I'})$. If $(m,\underline{\lambda},\underline{\lambda}',\boldsymbol{\mu},\boldsymbol{\nu})$ is of finite slope, then so is $(m,\underline{\lambda}^\vee,\underline{\lambda}^\vee,\check{\boldsymbol{\mu}},\check{\boldsymbol{\nu}})$ and while $U_\mathfrak{p}$ acts on the modified vector
$$
\tilde{m}:=\Pi_{\underline{\lambda}}^0\otimes\Pi_{\underline{\lambda}''}^0(m)
$$
via the scalar $\kappa_{\underline{\lambda}}\cdot\kappa_{\underline{\lambda}'}$, it acts on
$
(\tilde{m})^\vee
$
via the scalar
$$
\kappa_{\underline{\lambda}^\vee}\cdot\kappa_{\underline{\lambda}'^\vee}.
$$
Furthermore there exists an explicit non-zero constant $C\in E^\times$ with
$$
C\cdot (\tilde{m})^\vee=\Pi_{\underline{\lambda}^\vee}^0\otimes\Pi_{(\underline{\lambda}'^\vee)'}^0(m^\vee).
$$
\end{proposition}

\begin{proof}
The relation of the eigen values for the operator $U_\mathfrak{p}$ is an immediate consequence of our previous discussion. It remains to show that $\cdot^\vee$ commutes with the modification operator up to a constant. Applying formula \eqref{eq:twistedVpn} to the projection formula yields
$$
\left(\Pi_{\underline{\lambda}}^0(m)\right)^\vee=
\prod_{i=1}^{n-1}\prod_{\begin{subarray}cj=1\\j\neq i\end{subarray}}^{n}
\eta_n\cdot(
\lambda_i\cdot\absNorm(\mathfrak{p})^{1-j}\cdot
V_{\mathfrak{p},n+1-j}+ V_{\mathfrak{p},n-j})\cdot m^\vee=
$$
$$
\prod_{i=1}^{n-1}\prod_{\begin{subarray}cj=1\\j\neq i\end{subarray}}^{n}
\eta_n\cdot(
(\lambda_{n+1-i}^\vee)^{-1}\cdot\absNorm(\mathfrak{p})^{n+1-j-1}\cdot
V_{\mathfrak{p},n+1-j}+ V_{\mathfrak{p},n-j})\cdot m^\vee.
$$
Replacing $i$ with $n+1-i$ and $j$ with $n+1-j$ gives
$$
\prod_{i=1}^{n-1}\prod_{\begin{subarray}cj=1\\j\neq {i}\end{subarray}}^{n}
\eta_n\cdot(\lambda_{i}^\vee)^{-1}(
\lambda_{i}^\vee
\cdot\absNorm(\mathfrak{p})^{1-j}\cdot
V_{\mathfrak{p},j-1}+ V_{\mathfrak{p},j})\cdot m^\vee,
$$
proving the claim.
\end{proof}

\subsection{Contragredient cohomology}

Now we return to the global situation, i.e.\ $K$, $K'$ denote compact open subgroups of the finite adelic groups as before, of levels $I_n^{(m)}$ resp.\ $I_{n-1}^{(m)}$ at $\mathfrak{p}$. In this section the long Weyl element $w_n$ is always considered as embedded into the $\mathfrak{p}$ component, i.e.\
$$
w_n=w_n\otimes_{v\nmid\mathfrak{p}}{\bf1}_n\in\GL_n(\Adeles_k).
$$
Then we define $\iota:\GL_n(\Adeles_k)\to\GL_n(\Adeles_k)$ as
$$
g\mapsto w_n\cdot g^{-t}\cdot w_n,
$$
again with $w_n$ only at the $\mathfrak{p}$-component. We set $K^\vee:=\iota(K)$. Note that $\iota$ stabilizes the standard maximal compact subgroups as well as their connected components of the identity. Furthermore it stabilizes the center of $G_n(\RR)$, as well as its connected identity component. As $\iota$ is an idempotent, it also fixes Haar measures.

Note that if $M_{\boldsymbol{\mu}}$ is a representation of $\GL_n$, then $M_{\boldsymbol{\mu}}^\vee$ is isomorphic to the contragredient represenation $M_{\check{\boldsymbol{\mu}}}$. We fix such an isomorphism once and for all.

We have a diffeomorphism
$$
\iota_K:\mathscr X_n(K^\vee)\to\mathscr X_n(K),
$$
$$
G_n(\QQ)xK^\vee\mapsto G_n(\QQ)x^{-t}w_n K.
$$
It induces a pullback map on sheaves and we have a natural isomorphism
$$
\iota_K^\vee:\iota_K^*\underline{M}_{\boldsymbol{\mu}}(E)\to
\underline{M}_{\boldsymbol{\mu}}(E)^\vee,
$$
of sheaves on $\mathscr X_{n}(K^\vee)$, which is given on the sections by
$$
f\mapsto f^\vee.
$$
This morphism induces an isomorphism
$$
\iota_K^{\vee*}:=\iota_K^{\vee}\circ\iota_K^*:H_{*}^q(\mathscr X_{n}(K);\underline{M}_{\boldsymbol{\mu}}(E))\to
H_{*}^q(\mathscr X_{n}(K^\vee);\underline{M}_{\boldsymbol{\mu}}(E)^\vee),
$$
which twists the Hecke action at $\mathfrak{p}$ as in the previous section, i.e.\ we might identify the right hand side with
$$
H_{*}^q(\mathscr X_{n}(K);\underline{M}_{\boldsymbol{\mu}}(E))^\vee
$$
as $\mathcal H_{I^{(m)}}$-module. This is canonical, if we insist that $\iota_K^{\vee*}$ coinside with the map $\alpha\mapsto\alpha^\vee$. We have the fundamental property that for any $h\in G_n(\Adeles_\QQ)$
\begin{equation}
iTt_h^*
(\iota_K^{\vee*}
\alpha)=
\iota_{\iota(h)K\iota(h)^{-1}}^{\vee*}
(iTt_{\iota(h)}^*
\alpha).
\label{eq:contratranslate}
\end{equation}
We define the matrix
$$
\tilde{w}:=j(w_{n-1})\cdot w_n,
$$
which again lives only at $\mathfrak{p}$.

Now observe that translation by $\tilde{w}$ (at $\mathfrak{p}$) defines a diffeomorphism
$$
t_{\tilde{w}}:\mathscr X_n(j(w_{n-1})Kj(w_{n-1}))\to\mathscr X_n(K^\vee),
$$
$$
G_n(\QQ)xj(w_{n-1})Kj(w_{n-1})\mapsto G_n(\QQ)x\tilde{w} K^\vee.
$$
Therefore we get
\begin{equation}
\iota_{K'}\circ j=
j\circ \iota_{K}\circ t_{\tilde{w}}
\label{eq:wtilde}
\end{equation}
In particular the following the square
$$
\begin{CD}
H_{*}^q(\mathscr X_{n}(K);\underline{M}_{\boldsymbol{\mu}}(E))@>\iota_K^{\vee*}
>>H_{*}^q(\mathscr X_{n}(K^\vee);\underline{M}_{\boldsymbol{\mu}}(E)^\vee)\\
@Vj^*VV
@A{j^*\circ iTt_{\tilde{w}^{-1}}^*}AA\\
H_{*}^q(\mathscr X_{n-1}(K');\underline{j^*M}_{\boldsymbol{\mu}}(E))@>\iota_{K'}^{\vee*}
>>H_{*}^q(\mathscr X_{n-1}({K'}^\vee);(\underline{j^*M}_{\boldsymbol{\mu}}(E)^\vee)
\end{CD}
$$
is commutative.

\subsection{Contragredient matrix relation}

Finally to deduce the functional equation we need to establish some matrix relations. Consider the matrices
$$
n:=
\begin{pmatrix}
-1&0&\cdots&\cdots& 0  \\
-f&1&\ddots&     &\vdots\\
0&\ddots&\ddots&\ddots&\vdots\\
\vdots&  &\ddots&1&0\\
0&\cdots&0  &-f&-1\\
\end{pmatrix}\in\GL_n(k_{\mathfrak{p}})
$$
and
$$
n':=
d_{(x)}^{-1}\cdot
\begin{pmatrix}
1&f&f^2 &\cdots& f^{n-2}  \\
0&\ddots&\ddots&     &\vdots\\
\vdots&\ddots&\ddots&\ddots&f^2\\
\vdots&  &\ddots&\ddots&f\\
0&\cdots& \dots &0&1\\
\end{pmatrix}
\cdot
d_{(x)}
\in\GL_{n-1}(k_{\mathfrak{p}}).
$$
Finally set
$$
d=\diag(-x_\mathfrak{p},-1,\dots,-1,(-1)^nx_\mathfrak{p}^{-1})
\in\GL_{n-1}(k_\mathfrak{p}).
$$
Then we have, under the usual assumption $\nu_\mathfrak{p}(f)\geq m$,
\begin{lemma}\label{lem:inverseh}
For any $x_\mathfrak{p}\in k_\mathfrak{p}$ we have
\begin{equation}
j(w_{n-1}dn')\cdot\left(d_{(x_\mathfrak{p})}h^{(f)}\right)^{-t}w_n\cdot n\;=\;j(f^{n}\cdot{\bf1}_{n-1})f^{1-n}\cdot d_{((-1)^{n-1}x_\mathfrak{p}^{-1})}h^{(f)},
\label{eq:inverseh}
\end{equation}
with $w_{n-1}dn'w_{n-1}\in I_{n-1}^{(m)}$ and $\det(j(d)n')=1$, and $n\in I_{n}^{(m)}$.
\end{lemma}

We omit the proof, essentially an evaluation of a matrix identity.

\subsection{Proof of the functional equation}

The map
$$
\cdot^\vee:\Adeles_k^{(\infty)\times}\to\Adeles_k^{(\infty)\times}
$$
given by
$$
x\mapsto x^\vee:=(-1)^{n-1}x^{-1},
$$
where the $(-1)^{n-1}$ occurs only in the $\mathfrak{p}$-component, induces an involution
$$
\cdot^\vee:C(K(\mathfrak{p}^\infty))\to C(K(\mathfrak{p}^\infty)),
$$
and also an involution $\cdot^\vee$ on $C(\mathfrak{p}^\infty)$, which commutes with the covering map \eqref{eq:Cfcover}.

\begin{theorem}\label{thm:functionaleq}
Let
$$
\lambda\in
H_{\rm c}^{rb_n}(\mathscr X_n^{\ad}(K); \underline{M}_{\boldsymbol{\mu}}(E))\otimes_E
H_{\rm c}^{rb_{n-1}}(\mathscr X_{n-1}^{\ad}(K'); \underline{M}_{\boldsymbol{\nu}}(E))
$$
be an finite slope eigen class for the modified Hecke operator $U_\mathfrak{p}$ with eigen value $\kappa_\lambda\in E$. Then we have the functional equation
$$
(\mu_\lambda(x))^\vee=
\mu_{\lambda^\vee}(x^{\vee}).
$$
\end{theorem}

By composing the functional equation with the projection $\tau_\nu$ we get the explicit identity
\begin{equation}
(\tau_\nu(\mu_\lambda(x)))^\vee=
\tau_{-\nu}(\mu_{\lambda^\vee}(x^{\vee})).
\label{eq:explicitfunctional}
\end{equation}
Furthermore this functional equation is compatible with the complex functional equation, as the involution we defined is compatible with the notion of automorphic contragredient representations, and also preserves cohomological vectors, i.e.\ our involution is in particular compatible with our construction of cohomology classes, up to the explicit constant $C$ in Proposition \ref{prop:contragredient}.

\begin{proof}
Fix $x\in\Adeles_k^{(\infty)\times}$ and any nontrivial $\mathfrak{p}$-power $\mathfrak{f}$.
Write
$$
\zeta_\lambda,\zeta_{\lambda'}\in E^\times
$$
for the eigen values of $T_n^{\nu_{\mathfrak{p}}(\mathfrak{f})}\otimes 1$ resp.\ $1\otimes T_{n-1}^{\nu_{\mathfrak{p}}(\mathfrak{f})}$. Then by the Hecke relation \eqref{eq:twistedVp} we get an identity
\begin{equation}
\kappa_\lambda(\mathfrak{f})=\zeta_\lambda^{1-n}\cdot\zeta_{\lambda'}^{-n}\cdot\kappa_{\lambda^\vee}(\mathfrak{f}).
\label{eq:inversekappa}
\end{equation}
Now we have by the definitions, using the relation \eqref{eq:contratranslate} once,
$$
\mu_{\lambda^\vee}(x+\mathfrak{f})=
\kappa_{\lambda^\vee}(\mathfrak{f})\cdot
(iTt_{d_{(x)}f t_{(f)}}^*)
(s_1^*(iT_{h^{(f)}})\otimes 1)
{\mathscr B}_{h^{(f)},x}^{
K^\vee,
K'^\vee}
(\iota_K^{\vee*}\otimes\iota_{K'}^{\vee*}\lambda)=
$$
$$
\sum_{\varepsilon}
\kappa_{\lambda^\vee}(\mathfrak{f})\cdot
\!\!\!\!\!\!\!\!\!\!\!\!\!\!\!\!\!\!\!\!\!\!\!\!\!\!\!\!\!\!\!\!
\int\limits_{\mathscr X_{n-1}(d_{(x)}t_{(f)}K^\vee(\mathfrak{f})t_{(f)}^{-1}d_{(x)}^{-1})[1]}
\!\!\!\!\!\!\!\!\!\!\!\!\!\!\!\!\!\!\!\!\!\!\!\!\!\!\!\!\!\!
(iTt_{d_{(x)}f t_{(f)}}^*)
(s_1^*(\iota_{\iota(h^{(f)})K\iota(h^{(f)})^{-1}}^{\vee*}(iTt_{\iota(h^{(f)})}\alpha_\varepsilon))\cup
\iota_{K'}^{*\vee}(\ad^*\beta_\varepsilon),
$$
where $K^\vee(\mathfrak{f})$ is defined mutatis mutandis like $K(\mathfrak{f})$, using $K^\vee$ and $K'^\vee$ instead of $K$ and $K'$. As the embeddings $j$ and $s_1$ commute with the operator $\iota^{\vee*}$ up to translation by $\tilde{w}$ (cf. \eqref{eq:wtilde}), we deduce, again by \eqref{eq:contratranslate}, that for each $\varepsilon$ the above integrand equals
$$
\iota_{d_{(x)}f t_{(f)}\tilde{w}(\iota(h^{(f)})K\iota(h^{(f)})^{-1}\tilde{w}^{-1})\cap K')(d_{(x)}f t_{(f)})^{-1}}^{\vee*}
(iTt_{\iota(d_{(x)}f t_{(f)})}^*)
(s_1^*(iTt_{\tilde{w}\iota(h^{(f)})}^*\alpha_\varepsilon)\cup
\ad^*\beta_\varepsilon).
$$
We have
$$
\iota(
d_{(x)}t_{(f)}K^\vee(\mathfrak{f})t_{(f)}^{-1}d_{(x)}^{-1})=
\iota(d_{(x)}t_{(f)})(j^{-1}(\tilde{w}\iota(h^{(f)})K\iota(h^{(f)})^{-1}\tilde{w}^{-1})\cap K')\iota(t_{(f)}^{-1}d_{(x)}^{-1}),
$$
and as $\iota$ fixes Haar measures, we conclude that
$$
\mu_{\lambda^\vee}(x+\mathfrak{f})=\iota_{\iota(t_{(f)}d_{(x)})(j^{-1}(\tilde{w}\iota(h^{(f)})K\iota(h^{(f)})^{-1}\tilde{w}^{-1})\cap K')\iota(d_{(x)}^{-1}t_{(f)}^{-1})}^{\vee*}
$$
$$
\sum_{\varepsilon}
\kappa_{\lambda^\vee}(\mathfrak{f})\cdot
\!\!\!\!\!\!\!\!\!\!\!\!\!\!\!\!\!\!\!\!\!\!\!\!\!\!\!\!\!\!\!\!\!\!\!\!\!\!\!\!\!\!\!\!\!\!\!\!\!\!\!\!\!\!\!\!\!\!\!\!\!\!
\int\limits_{\mathscr X_{n-1}(
\iota(t_{(f)}d_{(x)})(j^{-1}(\tilde{w}\iota(h^{(f)})K\iota(h^{(f)})^{-1}\tilde{w}^{-1})\cap K')\iota(d_{(x)}^{-1}t_{(f)}^{-1})
)[1]}
\!\!\!\!\!\!\!\!\!\!\!\!\!\!\!\!\!\!\!\!\!\!\!\!\!\!\!\!\!\!\!\!\!\!\!\!\!\!\!\!\!\!\!\!\!\!\!\!\!\!\!\!\!\!\!\!\!\!\!\!\!\!
(iTt_{\iota(t_{(f)}f d_{(x)})}^*)
(s_1^*(iTt_{\tilde{w}\iota(h^{(f)})}^*\alpha_\varepsilon)\cup
\ad^*\beta_\varepsilon).
$$
We have the matrix relation
\begin{equation}
\iota(t_{(f)}f)\cdot f^{n}=
w_{n-1}f^{-1}t_{(f)}^{-1}w_{n-1}\cdot f^{n}=
t_{(f)}f.
\label{eq:inverseft}
\end{equation}
Setting in the notation of Lemma \ref{lem:inverseh}
$$
k':=w_{n-1}j(dn')w_{n-1}\in I_{n-1}^{(m)},
$$
$$
k:=n\in I_n^{(m)},
$$
the identity \eqref{eq:inverseh} reads
$$
j(k')\cdot j(\iota(d_{(x)}))\tilde{w}\iota(h^{(f)})\cdot k=
j(f^n\cdot{\bf 1}_{n-1})\cdot d_{((-1)^{n-1}x^{-1})}h^{(f)}\cdot f^{1-n}.
$$
Therefore
$$
\alpha_\varepsilon=iTt_{k}^*\alpha_\varepsilon,
$$
and similarly for $\beta_\varepsilon$ and $k'^{-1}$, we deduce that
$$
\mu_{\lambda^\vee}(x+\mathfrak{f})=
\iota_{t_{(f)}d_{(x^\vee)}(j^{-1}(h^{(f)}K(h^{(f)})^{-1})\cap K')d_{(x^\vee)}^{-1}t_{(f)}^{-1}}^{\vee*}
$$
$$
\sum_{\varepsilon}
\kappa_{\lambda^\vee}(\mathfrak{f})\cdot
\!\!\!\!\!\!\!\!\!\!\!\!\!\!\!\!\!\!\!\!\!\!\!\!\!\!\!\!\!\!\!\!\!\!\!\!\!\!\!\!\!\!\!\!\!\!\!\!\!\!\!\!\!\!\!\!\!
\int\limits_{\mathscr X_{n-1}(
t_{(f)}d_{(x^\vee)}(j^{-1}(h^{(f)}K(h^{(f)})^{-1})\cap K')d_{(x^\vee)}^{-1}t_{(f)}^{-1}
)[1]}
\!\!\!\!\!\!\!\!\!\!\!\!\!\!\!\!\!\!\!\!\!\!\!\!\!\!\!\!\!\!\!\!\!\!\!\!\!\!\!\!\!\!\!\!\!\!\!\!\!\!\!\!\!\!\!\!
(iTt_{\iota(t_{(f)}f)f^{n} d_{(x^\vee)}}^*)
(s_1^*(iTt_{h^{(f)}f^{1-n}}^*\alpha_\varepsilon)\cup
\ad^*iTt_{f^{-n}}^*\beta_\varepsilon).
$$
By the matrix relation \eqref{eq:inverseft} and the definition of the Hecke operators $T_{n}\otimes 1$ and $1\otimes T_{n-1}$, relation \eqref{eq:inversekappa} shows the claim.
\end{proof}

\providecommand{\bysame}{\leavevmode\hbox to3em{\hrulefill}\thinspace}

$$
\underline{\;\;\;\;\;\;\;\;\;\;\;\;\;\;\;\;\;\;\;\;\;\;\;\;\;\;\;\;\;\;}
$$\ \\
Karlsruher Institut f\"ur Technologie, Fakult\"at f\"ur Mathematik, Institut f\"ur Algebra und Geometrie, Kaiserstra\ss{}e 89-93, 76133 Karlsruhe, Germany.\\
{januszewski@kit.edu}

\end{document}